\documentclass[12pt]{article}
\usepackage{amsmath}
\usepackage{amssymb}
\usepackage{amsthm}
\usepackage{latexsym}
\usepackage{color,fancyhdr,lastpage,graphicx}
\usepackage{subfigure, epsfig, epsf}
\usepackage{xspace}
\usepackage{setspace}
\usepackage{bbm}

\graphicspath{{thesis/Figures/}}

\theoremstyle{plain}

\newtheorem{lemma}{Lemma}[section]

\newtheorem{theorem}[lemma]{Theorem}
\newtheorem{proposition}[lemma]{Proposition}
\newtheorem{corollary}[lemma]{Corollary}
\theoremstyle{remark}
\newtheorem{rem}[lemma]{Remark}

\def\bb{\begin{color}{black}}
\def\bg{\begin{color}{black}}
\def\br{\begin{color}{black}}
\def\eg{\end{color}}
\def\er{\end{color}}
\def\eb{\end{color}}

\def\PP{\mathbb{P}}
\def\F{\mathbb{F}}
\def\bF{\bar{\mathbb{F}}}
\def\GG{\mathbb{G}}
\def\bGG{\bar{\mathbb{G}}}
\def\ve{{\varepsilon}}
\def\le{\leqslant}
\def\pd{\partial}

\def\ge{\geqslant}

\def\es{\emptyset}

\def\E{{\mathbb E}}
\def\O{{\Omega}}

\def\R{{\mathbb R}}
\def\C{{\mathbb C}}
\def\N{{\mathbb N}}

\def\Z{{\mathbb Z}}

\def\a{{\alpha}}
\def\b{{\beta}}
\def\d{{\delta}}
\def\D{{\Delta}}
\def\t{{\tau}}

\def\g{{\gamma}}
\def\G{{\Gamma}}
\def\s{{\sigma}}
\def\l{{\lambda}}

\def\th{{\theta}}
\def\Th{{\Theta}}

\def\cS{{\cal S}}

\def\cF{{\cal F}}

\def\cD{{\cal D}}

\def\q{\quad}
\def\id{\operatorname{id}}
\def\re{\operatorname{Re}}
\def\im{\operatorname{Im}}
\def\cp{\operatorname{cap}}

\def\dist{\operatorname{dist}}
\def\fing{\operatorname{finger}}

\def\gap{\operatorname{gap}}

\def\<{\langle}
\def\>{\rangle}

\def\ua{\uparrow}
\def\sse{\subseteq}
\def\sm{\setminus}

\renewcommand\epsilon{\ve}
\begin{document}
\bibliographystyle{plain}

\begin{center}
\LARGE \textbf{Hastings--Levitov Aggregation in the Small-Particle Limit}

\vspace{0.2in}

\large {\bfseries James Norris
\footnote{Statistical Laboratory, Centre for Mathematical Sciences,
  Wilberforce Road, Cambridge, CB3 0WB, UK}
\footnote{Research supported by EPSRC grant EP/103372X/1}
\& Amanda Turner
\footnote{Department of Mathematics and Statistics,
Lancaster University,
Lancaster,
LA1 4YF, UK}}

\vspace{0.2in}
\small
\today

\end{center}
\begin{abstract}
We establish some scaling limits for a model of planar aggregation.
The model is described by the composition of a sequence of independent and identically
distributed random conformal maps, each corresponding to the addition of one particle.
We study the limit of small particle size and rapid aggregation.
The process of growing clusters converges, in the sense of Carath\' eodory, to an inflating disc.
A more refined analysis reveals, within the cluster, a tree structure of branching fingers, whose radial component increases
deterministically with time.
The arguments of any finite sample of fingers, tracked \bg inwards\eb, perform coalescing Brownian motions.
The arguments of any finite sample of gaps between the fingers, tracked \bg outwards\eb,
also perform coalescing Brownian motions.
These properties are closely related to the evolution of harmonic measure on the boundary of the cluster,
which is shown to converge to the Brownian web.
\end{abstract}
\vspace{0.2in}


\section{Introduction}
Consider an increasing sequence $(K_n:n\ge0)$ of compact subsets of the complex plane,
starting from the closed unit disc $K_0$ centred at $0$.
Set $D_n=(\C\cup\{\infty\})\sm K_n$ and assume that $D_n$ is simply connected.
Write $K_n$ as a disjoint union $K_0\cup P_1\cup\dots\cup P_n$.
Think of $K_n$ as a cluster formed by attaching a sequence of particles $P_1,\dots,P_n$ to $K_0$.
By the Riemann mapping theorem, there is a unique normalized conformal map $\Phi_n:D_0\to D_n$.
Here, by normalized we mean that $\Phi_n(z)=e^{c_n}z+O(1)$ as $|z|\to\infty$ for some $c_n\in\R$.
By a conformal map $D_0\to D_n$ we always mean a conformal isomorphism, in particular a bijection.
The constant $c_n$ is the logarithmic capacity $\cp(K_n)$ and the sequence $(c_n:n\ge0)$ is increasing.
We can write $\Phi_n=F_1\circ\dots\circ F_n$, where each $F_n$ is a normalized conformal map from $D_0$
to a neighbourhood of $\infty$ in $D_0$.
Moreover, any sequence $(F_n:n\in\N)$ of such conformal maps is associated to such a sequence
of sets $(K_n:n\ge0)$ in this way.

Hasting and Levitov \cite{HL} introduced a family of models for random planar growth, indexed by a parameter $\a\in[0,2]$.
We shall study a version of the case $\a=0$, which may be described as follows.
Let $P$ be a non-empty and connected subset of $D_0$, having $1$ as a limit point. Set $K=K_0\cup P$ and $D=(\C\cup\{\infty\})\sm K$.
Assume that $K$ is compact and that $D$ is simply connected.
We think of $P$ as a particle attached to $K_0$ at $1$.
For example, $P$ could be a disc of diameter $\d$ tangent to $K_0$ at $1$, or a line segment $(1,1+\d]$.
We sometimes allow the case where $P$ has other limit points in $K_0$, for example $P=\{z\in D_0:|z-1|\le\d\}$,
but always give $1$ the preferred status of attachment point.
Write $F$ for the unique normalized conformal map $D_0\to D$ and set $c=\cp(K)$.
\bg
We assume throughout that $F$ extends continuously to the closure $\bar D_0$. This is known to hold if and only if
$K$ is locally connected.
\eg
Let $(\Th_n:n\in\N)$ be a sequence of independent random variables, each uniformly distributed on $[0,2\pi)$.
Define for $n\ge1$
\begin{equation}\label{PHIN}
F_n(z)=e^{i\Th_n}F(e^{-i\Th_n}z),\q \Phi_n=F_1\circ\dots\circ F_n.
\end{equation}
Write $(K_n:n\in\N)$ and $(P_n:n\in\N)$ for the associated sequences of random clusters and particles.

Note that $\cp(K_n)=cn$.
Note also that $P_{n+1}=\Phi_n(e^{i\Th_{n+1}}P)$.
Since harmonic measure is conformally invariant, conditional on $K_n$,
the random point $\Phi_n(e^{i\Th_{n+1}})$ at which
$P_{n+1}$ is attached to $K_n$ is distributed on the boundary of $K_n$
according to the normalized harmonic measure from infinity.
However $P_{n+1}$ is not a simple copy of $P$, as would be natural in a model of diffusion limited
aggregation, but is distorted\footnote{If we suppose (unrealistically) that $\Phi_n'$ is nearly constant on the scale of $P$, then a rough
compensation for the distortion would be achieved by replacing $P$ in the definition of $P_{n+1}$
by a scaled copy of diameter $\d_{n+1}=|\Phi_n'(e^{i\Th_{n+1}})|^{-1}\d$.
More generally, we could interpolate between these models by taking $\d_{n+1}=|\Phi_n'(e^{i\Th_{n+1}})|^{-\a/2}\d$ for some fixed $\a\in[0,2]$.
This is the family proposed by Hastings and Levitov.}
by the map $\Phi_n$.

We obtain results which describe the limiting
behaviour of the growing cluster when the basic particle $P$ has small diameter $\d$,
identifying both its overall shape and the distribution of random structures of `fingers' and `gaps'.
Some of these results are stated in Section \ref{illus}.
The results are accompanied by illustrations of typical clusters for certain cases of the model.
We need some basic estimates for conformal maps, which are derived in Section \ref{SBE}.
A simplifying feature of the case $\a=0$ is that fact that, for $\G_n=\Phi^{-1}_n$, the process $(\G_n(z):n\ge0)$ is Markov, for all $z\in D_0$.
This enables us to do
a fluid limit analysis in Section \ref{FLA} for the random flows $\G_n$ as the particles become small,
showing that after adding $n$ particles, the cluster fills out a disc of radius $e^{cn}$, with only small holes.
In Section \ref{HMLP}, we obtain some further estimates which show that the harmonic measure from infinity
on the boundary of the cluster is concentrated near the circle of radius $e^{cn}$ and spread out evenly around the circle.
We also bound the distortion of individual particles. Section \ref{WC} reviews some weak approximation theorems for the
coalescing Brownian flow from \cite{NT1}. These are then applied to the flow of harmonic measure on the cluster
boundary in Section \ref{SLAM}. In conjunction with the results of Section \ref{HMLP}, this finally allows us to
identify the weak limit of the fingers and gaps.

\section{Review of related work}\label{RRW}
\bb
There has been strong interest in models for the random growth of clusters over the last 50 years. Early models were often set up
on a lattice, such as the Eden model \cite{Eden}, Witten and Sander's diffusion limited aggregation (DLA) \cite{W+S},
and the family of dielectric breakdown models of Niemeyer et al. \cite{NPW}. The primary interest in these and other related processes has been in the
asymptotic behaviour of large clusters.

Computational investigations of these lattice based models have revealed
structures, of fractal type, which in some cases resemble natural
phenomena. However, such investigations have also shown sensitivity to details of implementation,
in particular to the geometry of the underlying lattice. For example, in \cite{BBRT} and \cite{MBRS} different fractal dimensions are obtained for DLA constructed with different lattice dependencies. This suggests that lattice-based models may not be the most effective way to describe these physical structures. In addition, lattice based models have proved difficult to analyse. There are few notable
mathematical results, with the exception of Kesten's 1987 growth estimate for DLA \cite{Kesten}, and there is much that remains to be understood about the large-scale behaviour of these models and in particular about the structure of fingers which is characteristically observed.

In 1998, Hastings and Levitov \cite{HL} formulated a family of continuum growth models
in terms of sequences of iterated conformal maps, indexed by a parameter $\a\in[0,2]$.
They argue, by comparing local growth rates, that their models share features with lattice
dielectric breakdown in the range $\a\in[1,2]$, so that $\a=1$ corresponds to the Eden model, and $\a=2$ to DLA.
Further exploration of this relation is discussed in the survey paper by Bazant and Crowdy \cite{B+C}.

The Hastings--Levitov family of models has been discussed extensively in the physics literature from a numerical point of view. In their original paper, Hastings and Levitov found experimental evidence of a phase transition at $\a=1$, and further studies can be seen in, for example, \cite{David} where estimates for the fractal dimensions of clusters are obtained, \cite{JLMP} where the multifractal properties of harmonic measure on the cluster are explored, and \cite{Hastings01} where the dependence of the fractal dimension on $\alpha$ is investigated.

Although this conformal mapping approach to planar random growth processes has proved more tractable than the lattice approach, there have been few rigorous mathematical results, particularly in the case $\a>0$.
Carleson and Makarov \cite{C+M}, in 2001, obtained a growth estimate for a deterministic analogue
of the DLA model.
In 2005, Rohde and Zinsmeister \cite{RZ} considered
the case $\alpha=0$ in the Hastings--Levitov family.
They established a long-time scaling limit, for fixed particle size and showed that the limit law was supported
on clusters of dimension $1$.
They also gave estimates for the dimension of the limit sets in the case of general $\alpha$,
and discussed limits of deterministic variants. Recently, Johansson Viklund, Sola and Turner \cite{JST} studied an anisotropic version of the Hastings--Levitov model in the $\a=0$ case, and established deterministic scaling limits for the macroscopic shape and evolution of harmonic measure on the cluster boundary.

In this paper, we also consider the case $\alpha=0$ but in the limiting regime where the particle
diameter $\d$ becomes small and where the size of the cluster is of order $1$ or larger. We obtain a precise description of the macroscopic shape and growth dynamics of these clusters, as well as a fine scale description of the underlying branching structure. In the process of obtaining these results, we show that the evolution of harmonic measure on the cluster boundary converges to the coalescing Brownian flow, also known as the Brownian web \cite{FINR}.
An early version of some parts of the present paper, along with its companion paper \cite{NT1}, appeared in \cite{NT}.
\eb

\section{Statement of results}\label{illus}
We state here our main results on the shape and structure of the Hastings--Levitov cluster.
Our main result on the harmonic measure flow, which cannot be stated so directly, is Theorem \ref{HMFL}.
For simplicity, we assume in this section that the basic particle $P$ is either a slit $(1,1+\d]$
or a disc $\{|z-1-\d/2|\le\d/2\}$, and that $\d\in(0,1/3]$.
\bg
We shall prove our results under some general conditions (\ref{D13}),(\ref{D14}),(\ref{HMCO}) on the basic particle $P$,
which can be readiliy checked for the slit and disc models.
We shall see that under one of these conditions (\ref{D13}) the logarithmic capacity $c=\cp(K)=\log F'(\infty)$ of $K$ satisfies $\d^2/6\le c\le 3\d^2/4$.
\eg
Our first result expresses that the cluster $K_n$ is contained in a disc of approximate radius
$e^{cn}$ and fills out that disc with only small holes. Moreover, there is a rough correspondence between the time
at which a particle arrives and its distance from the origin.
\begin{theorem}\label{BALL}
Consider for $\ve\in(0,1]$ and $m\in\N$ the event $\O[m,\ve]$ specified by the following conditions:
for all $n\le m$ and all $n'\ge m+1$,
$$
|z-e^{cn+i\Th_n}|\le\ve e^{cn}\q\text{for all $z\in P_n$}
$$
and
$$
\dist(w,K_n)\le\ve e^{cn}\q\text{whenever $|w|\le e^{cn}$}
$$
and
$$
|z|\ge(1-\ve)e^{cm}\q\text{for all $z\in P_{n'}$}.
$$
Assume that \bg $\ve=\d^{2/3}(\log(1/\d))^8$ \eb and $m=\lfloor\d^{-6}\rfloor$. Then
$\PP(\O[m,\ve])\to1$ as $\d\to0$.
\end{theorem}
\bg This result is a special case of Theorem \ref{BALLG} below.
Note that $\O[m,\ve]$ is decreasing in $m$ and increasing in $\ve$. \eg
We have made some effort to maximise the power $2/3$ in this statement. It will be crucial later
that $2/3>1/2$. \bb We shall take particular interest in the case where $m$ is of order $\d^{-2}$ and in the case where $m$ is of order $\d^{-3}$,
when the diameter of the cluster $K_m$ is of order $1$ and $\d^{-1}$ respectively\eb.  We have
not attempted to optimise the power \bg $8$ \eb in the logarithm.

In Figure \ref{DLAfig}, we present some realizations of the cluster when $P$ is
a slit\footnote{\bb The normalized conformal map $G=F^{-1}: D \to D_0$ can be obtained in this case as $\phi^{-1} \circ g_1 \circ \phi$, where $\phi$ takes $D_0$ to the upper half plane $H_0$ by $\phi(z)=i(z-1)/(z+1)$ and $g_1(z)=\sqrt{(z^2 + t)/(1-t)}$ takes $H=H_0 \setminus (0,i \sqrt{t}]$ to $H_0$, where $t=\d^2/(2+\d)^2$. A straightforward calculation gives $c=c(\d)=-\log G'(\infty)=-\log(1-t) \asymp \d^2/4$. \eb
}
$(1,1+\d]$, for various values of $\d$.
\begin{figure}[p]
  \vspace{-20pt}
    \subfigure[The cluster after a few arrivals with $\d=1$.]
      {\label{DLAfew}
      \epsfig{file=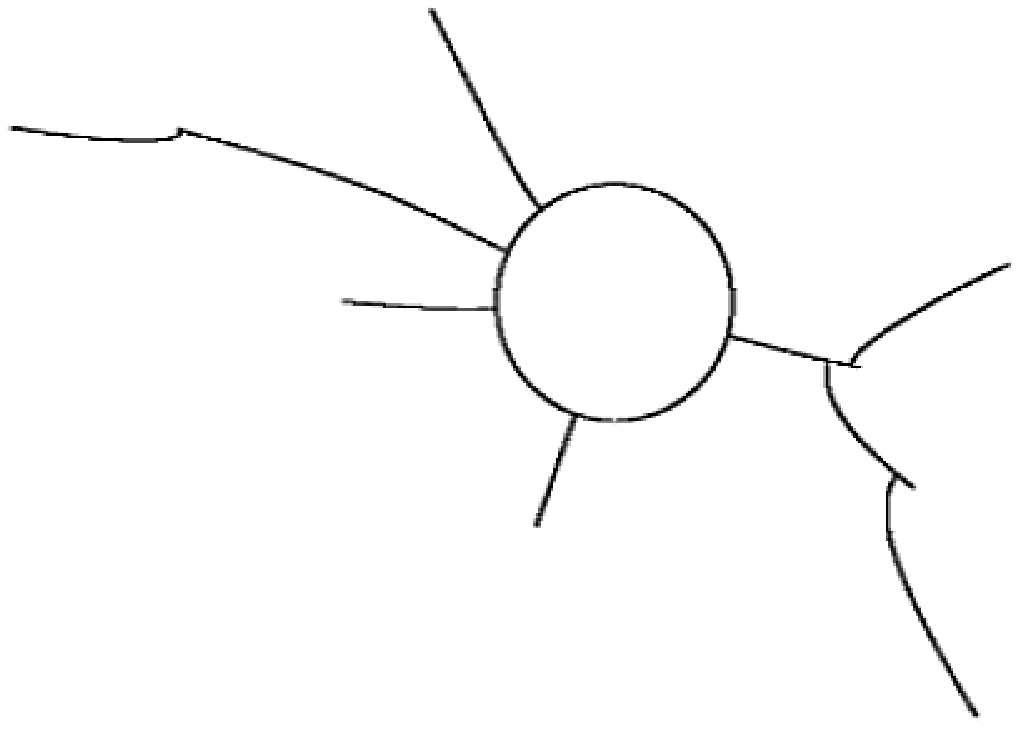,width=5.8cm}}
    \hfill
    \subfigure[The cluster after 100 arrivals with $\d=1$.]
      {\label{DLAmany}
      \epsfig{file=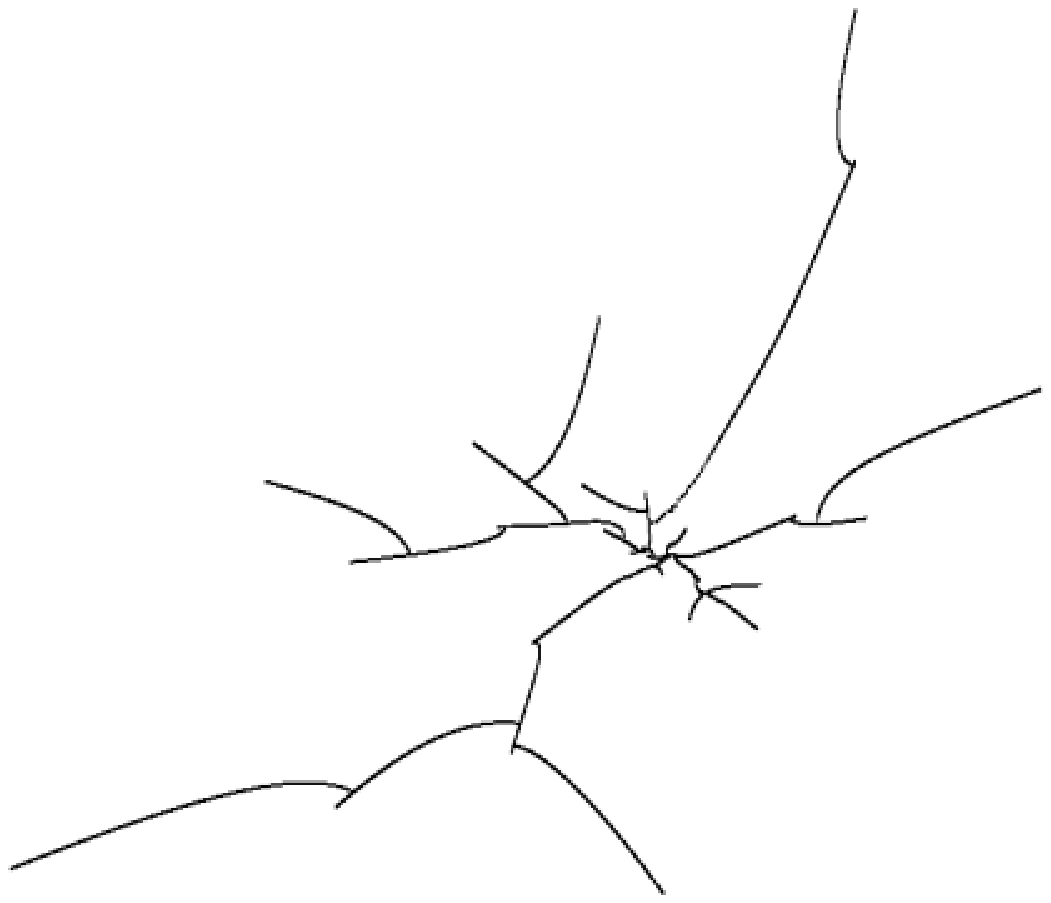,width=5.8cm}}
    \subfigure[The cluster after 800 arrivals with $\d=0.1$.]
      {\label{slit10}
      \epsfig{file=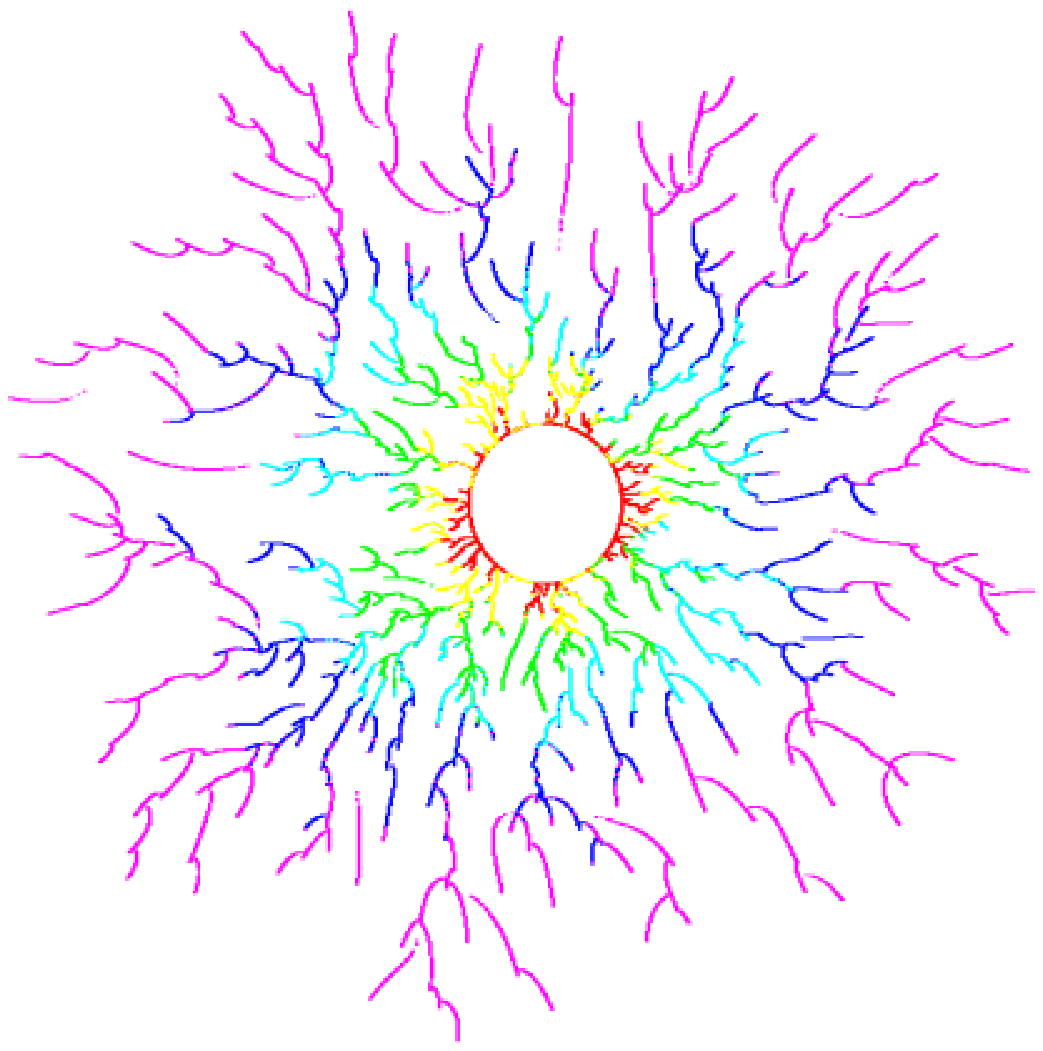,width=5.8cm}}
    \hfill
    \subfigure[The cluster after 5000 arrivals with $\d=0.04$.]
      {\label{slit25}
      \epsfig{file=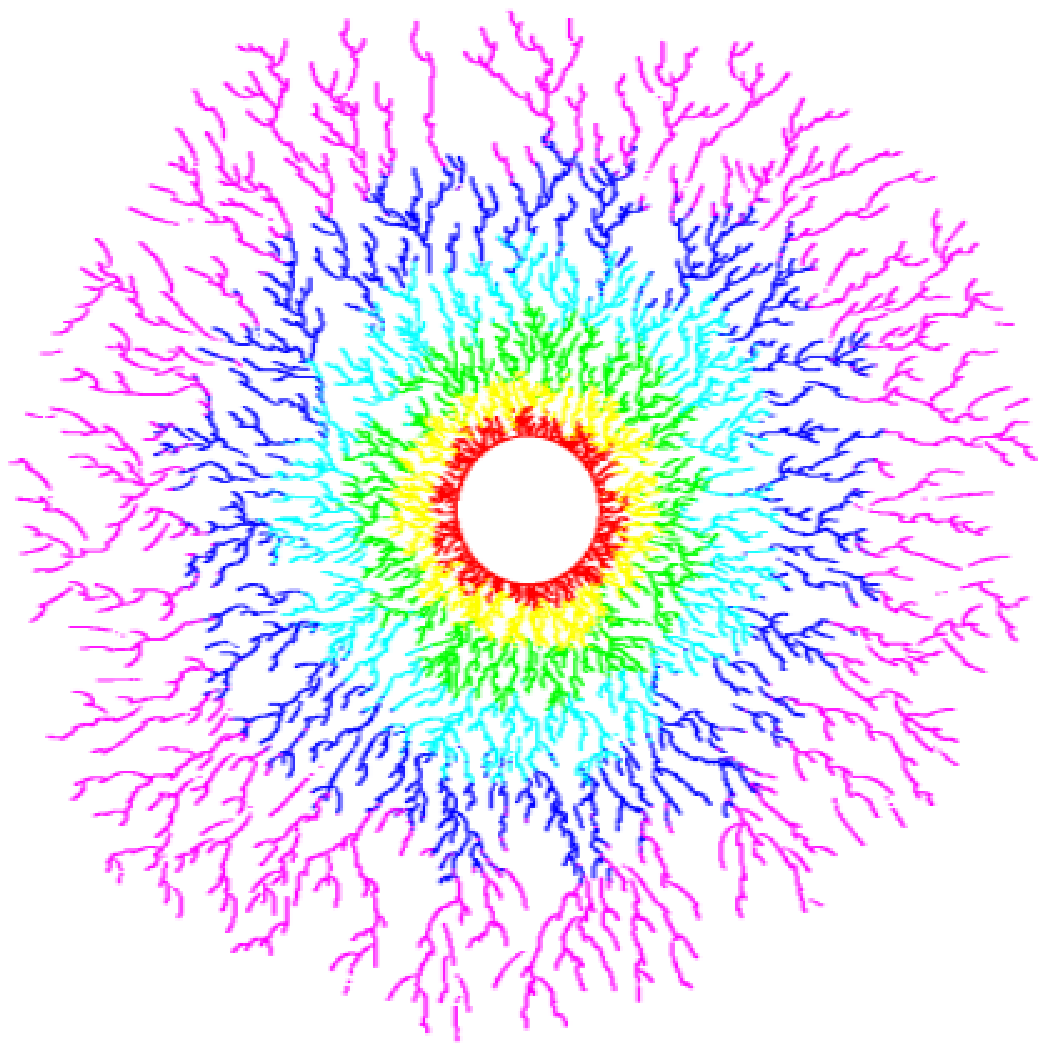,width=5.8cm}}
    \subfigure[The cluster after 20000 arrivals with $\d=0.02$.]
      {\label{slit50}
      \epsfig{file=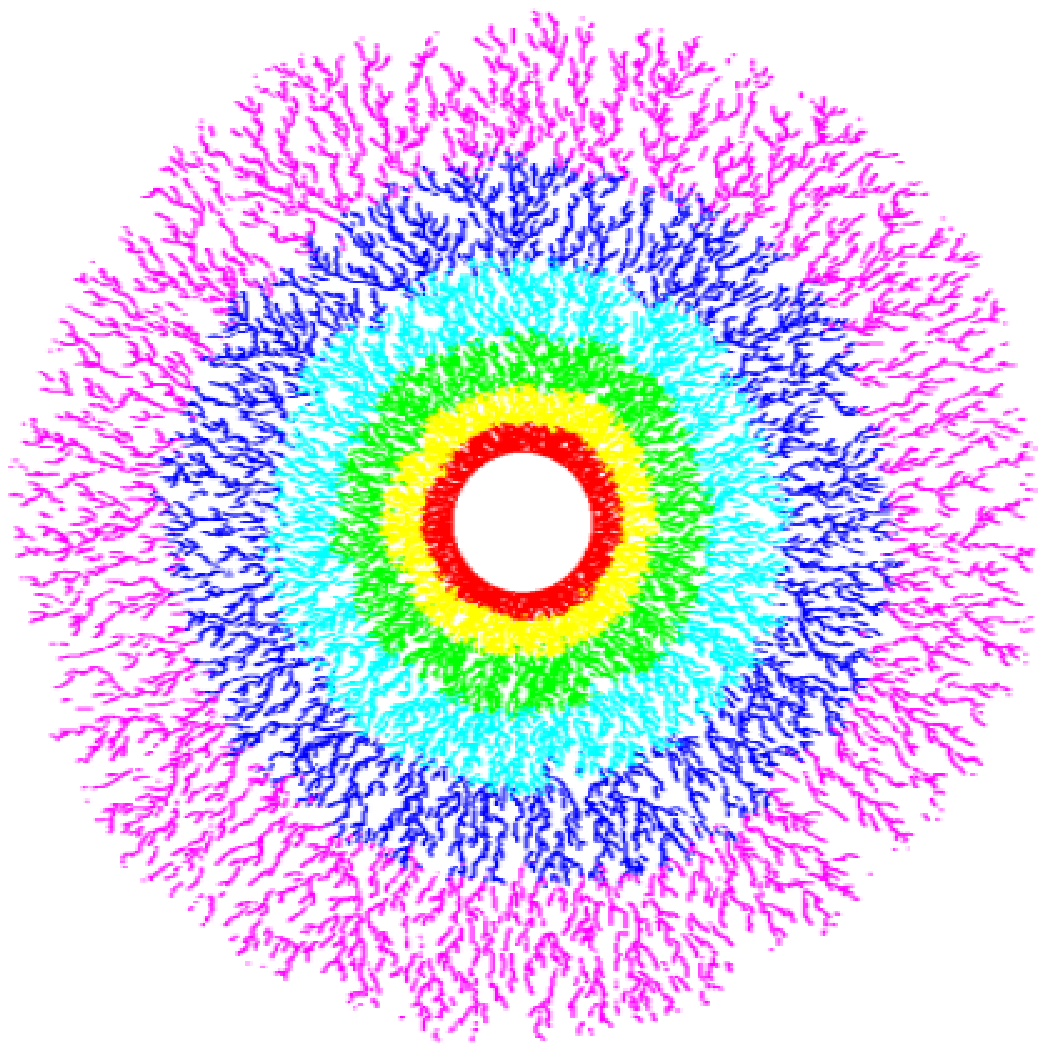,width=5.8cm}}
    \hfill
    \subfigure[Trajectories of $\tilde\G_n(e^{2\pi ix})/(2\pi i)$ for $\d=0.02$, with $t=n/10^6$.]
      {\label{flow50}
      \epsfig{file=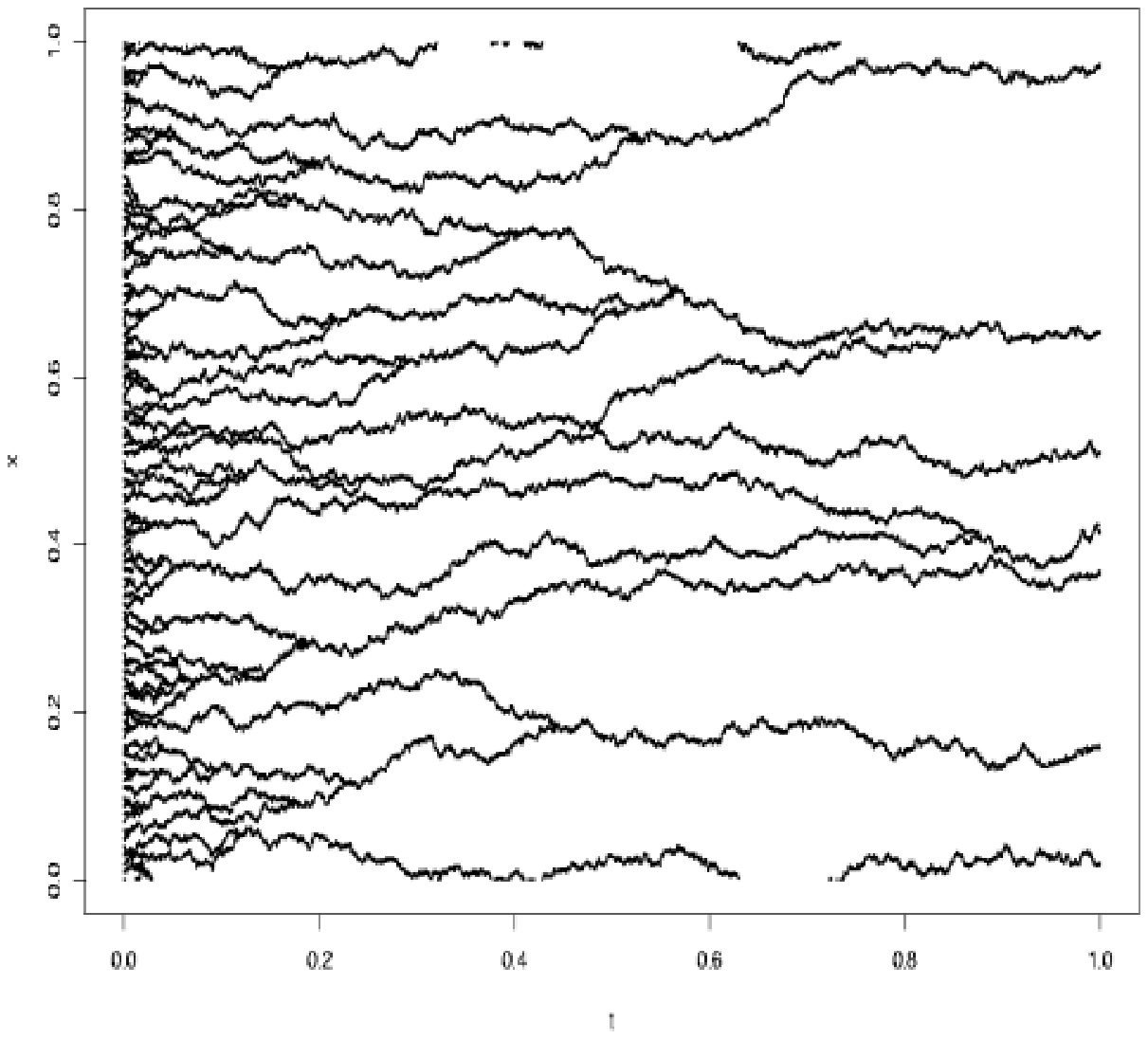,width=5.8cm}}
  \caption{\textsl{The slit case of HL$(0)$}}
  \label{DLAfig}
\end{figure}
We observe in Figure \ref{DLAmany}, when $\d=1$, that incoming particles are markedly distorted
and that particles arriving later tend to be larger. This effect is diminished
when we examine smaller values of $\d$. In Figure \ref{slit50}, the cluster is a rough disc,
as predicted by Theorem \ref{BALL} but with
some sort of internal structure. The colours label arrivals in different epochs, showing that
there is a close relationship between the time of arrival and the distance
from the origin at which a particle sticks, as in Theorem \ref{BALL}.
Figure \ref{flow50} focuses on the motion of points
on the boundary of the unit circle, under the inverse map $\G_n=\Phi_n^{-1}$
and over a longer timescale than for the other simulations.
This motion suggests the behaviour of coalescing Brownian motions,
which is confirmed in Theorem \ref{HMFL} below.

We now fix $N\in\N$ and state two results describing the internal geometry of the cluster $K_N$ in terms of coalescing Brownian motions,
which will follow from Theorems \ref{BALL} and \ref{HMFL}.
Define
$$
\tilde K_n=\{z\in\C:e^z\in K_n\},\q \tilde D_n=\{z\in\C:e^z\in D_n\}
$$
and determine $\rho=\rho(P)\in(0,\infty)$ by
$$
\frac\rho{2\pi}\int_0^{2\pi}(g(\th)-\th)^2d\th=1
$$
where $g$ is the unique continuous map $(0,2\pi)\to(0,2\pi)$ such that $g(\pi)=\pi$ and $G(e^{i\th})=e^{ig(\th)}$ for all $\th$.
We shall show in Proposition \ref{TGES} that $\d^{-3}/C\le\rho\le C\d^{-3}$ for an absolute constant $C<\infty$.
Note that $K_N$ has a natural notion of ancestry for its constituent particles: we say that $P_k$
is the parent of $P_{n+1}$ if $\Phi_n(e^{i\Th_{n+1}})\in P_k$.
This notion is inherited by the covering cluster $\tilde K_N$ and will allow us to identify path-like structures
within the cluster.
For $\re(z)\ge0$, denote by $\tilde P_0(z)$ the closest particle to $z$ in $\tilde K_N$,
and recursively denote by $\tilde P_m(z)$ the parent of $\tilde P_{m-1}(z)$ until $m=m(z)$ when $\tilde P_{m(z)}(z)$ is
attached to the imaginary axis, at $a(z)$ say. Consider the compact set
$$
\fing(z)=\{a(z)\}\cup\bigcup_{m=0}^{m(z)}\tilde P_m(z).
$$
We shall describe also the structure of the complementary set $\tilde D_N$, using a choice of paths in this set.
The notion of ancestry is not available, so we look instead for paths in the gaps which lead mainly outwards,
that is to the right in the logarithmic picture.
In order to enforce this outwards property, we impose a condition of minimal length, which requires a suitable completion
of the set of paths.
By a {\em gap path} we mean a rectifiable path $(p_\t)_{\t\ge0}$ in $\C$, parametrized by arc length, such that
$\re(p_\t)\to\infty$ as $\t\to\infty$ and
such that, for some continuous map $h:[0,\infty)\times[0,1]\to\C$ and for all $\t\ge0$, we have $p_\t=h(\t,1)$
and $h(\t,t)\in\tilde D_N$ for all $t\in[0,1)$.
For $R>0$, define $L_R(p)=\inf\{\t\ge0:\re(p_\t)=R\}$.
Write $p_0(z)$ for the closest point to $z$ which is not in the interior of $\tilde K_N$.
Since $\tilde D_N$ is simply connected and $K_N$ is compact, there exists a unique gap path $p(z)$ starting from $p_0(z)$ and minimizing $L_R(p)$
over all gap paths starting from $p_0(z)$, for all sufficiently large $R$.
\bg
The path $p(z)$ may be thought of as a long piece of thread outside the cluster, with one end attached to $p_0(z)$ and drawn tight by pulling from the right.
\eg
Set
$$
\gap(z)=\{p_\t(z):\t\ge0\}.
$$
Note that, by minimality, for all $\t_1,\t_2\ge0$ with $\t_1<\t_2$
and such that the open line segment $I=(p_{\t_1}(z),p_{\t_2}(z))$ is contained in $\tilde D_N$, we have $p_\t(z)\in I$ for all $\t\in(\t_1,\t_2)$.
\bb These definitions are illustrated in Figure \ref{fingergapfig}. Both fingers and gaps depend implicitly on $N$, although we have suppressed this in the notation.
\eb

\begin{figure}[ht]
\centering
\epsfig{file=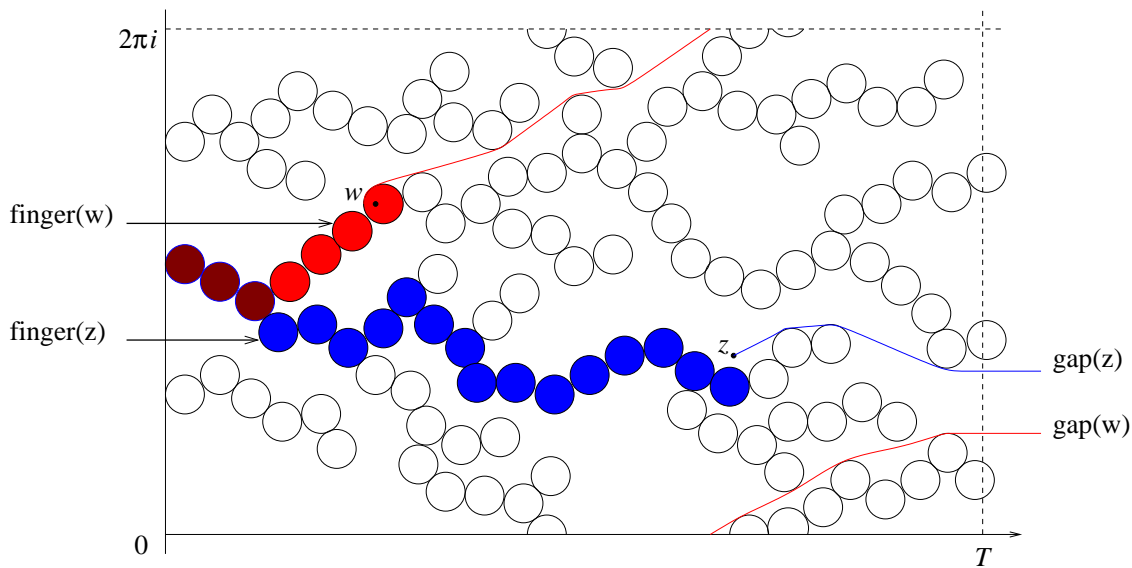,width=15cm}
\caption{\textsl{ Diagram illustrating fingers and gaps in $\tilde K_N$ (repeating periodically). This is only a representation and in general the particles will be distorted both by the conformal mapping and by the logarithmic transformation.}}
\label{fingergapfig}
\end{figure}

In order to capture the limiting fluctuations of the fingers and gaps we have to rescale.
We do this in two ways, defining horizontal and vertical scaling operators $\s$ and $\bar\s$ by
$$
\s(r+i\th)=(\d^*r,\th),\q \bar\s(r+i\th)=(r,\th/\sqrt{\d^*}),\q r\ge0,\q \th\in\R
$$
where $\d^*=(\rho c)^{-1}$. Note that $\d/C\le\d^*\le C\d$ for an absolute constant $C<\infty$. Also
$\bar\s=\s_{\d^*}\circ\s$ where $\s_{\d^*}$ is the diffusive scaling
$$
\s_{\d^*}(s,x)=(s/\d^*,x/\sqrt{\d^*}),\q s\ge0,\q x\in\R.
$$
The horizontal scaling identifies global random behaviour in the fingers and gaps over very long time scales, whereas
the vertical scaling identifies local fluctuations in the fingers and gaps while the size of the cluster is of order $1$.

Denote by $\cS$ the space of closed subsets of $[0,\infty)\times\R$, equipped with a local Hausdorff metric.
Define $\F,\GG:[0,\infty)\times\R\to\cS$ and $\bF,\bGG:[0,\infty)\times\R\to\cS$ by
$$
\F=\s\circ\fing\circ\,\s^{-1},\q \GG=\s\circ\gap\circ\,\s^{-1}
$$
$$
\bF=\bar\s\circ\fing\circ\,{\bar\s}^{-1},\q \bGG=\bar\s\circ\gap\circ\,{\bar\s}^{-1}.
$$
Thus, for $e=(s(e),x(e))$,
$$
\F(e)=\{\s(w):w\in\fing(s(e)/\d^*+ix(e))\},\q \GG(e)=\{\s(w):w\in\gap(s(e)/\d^*+ix(e))\}
$$
$$
\bF(e)=\{\bar\s(w):w\in\fing(s(e)+ix(e)\sqrt{\d^*})\},\q\bGG(e)=\{\bar\s(w):w\in\gap(s(e)+ix(e)\sqrt{\d^*})\}.
$$
We consider $\F(e),\bF(e),\GG(e),\bGG(e)$ as random variables in $\cS$.

We state first the long time result.
Fix $T>0$ and let $E$ be a finite subset of $[0,T]\times\R$.
Take $N=\lfloor \rho T\rfloor$ so that $K_N$ is approximately a disc of radius $e^{T/\d^*}$.
Denote by $\nu_E^P$ and $\eta_E^P$ the respective laws of $(\F(e):e\in E)$ and $(\GG(e):e\in E)$ on $\cS^E$.
Let $(B^e:e\in E)$ be a family of $2\pi$-coalescing Brownian motions, $B^e$ running backwards in time from $x(e)$ at time $s(e)$.
Thus $B^e=(B_t^e:0\le t\le s(e))$ and for all $e,e'\in E$, $B^e$ and $B^{e'}$ are independent
until (time running backwards) their difference is an integer multiple of $2\pi$, at which point it freezes.
Let $(W^e:e\in E)$ be a family of $2\pi$-coalescing Brownian motions,
with $W^e$ running forwards in time from $x(e)$ at time $s(e)$.
Denote by $\nu_E$ and $\eta_E$ the laws on $\cS^E$ of the families of random sets $(\{(t,B_t^e):0\le t\le s(e)\}:e\in E)$
and $(\{(t,W_{t\wedge T}^e):t\ge s(e)\}:e\in E)$.
\begin{theorem}\label{LONG}
We have $\nu_E^P\to\nu_E$ and $\eta_E^P\to\eta_E$ weakly on $\cS^E$ as $\d\to0$.
\end{theorem}

\bg
Thus, for small $\d$, we can construct on a common probability space, the cluster $K_N$ and
backwards and forwards $2\pi$-coalescing Brownian motions, such that the union
of fingers in $\tilde K_N$ starting from points $s(e)/\d^*+ix(e)$, $e \in E$ is, with probability close to $1$, close in Hausdorff metric
to the set $\bigcup_{e \in E} \{t/\d^*+iB^e_t: 0 \le t \le s(e) \}$,
and hence the union of fingers in $K_N$, starting from points $\exp(s(e)/\d^*+ix(e))$, $e \in E$
looks approximately like the set $\bigcup_{e \in E} \{\exp (t/\d^*+iB^e_t): 0 \le t \le s(e) \}$.
Similarly, the union of gaps in $K_N$, starting from points $\exp(s(e)/\d^*+ix(e))$, $e \in E$ looks
approximately like the set $\bigcup_{e \in E} \{\exp (t/\d^*+iW^e_{t\wedge T}): t \geq s(e) \}$.
A simulation of $\bigcup_{e \in E} \{\exp (t/\d^*+iB^e_t): 0 \le t \le s(e) \}$ and
$\bigcup_{e \in E} \{\exp (t/\d^*+iW^e_t): s(e) \le t \le T \}$ is shown in Figure \ref{globallimitfig}.
\eb

For the local result we take now $N=\lfloor c^{-1}T\rfloor$ so that $K_N$ is approximately a disc of radius $e^T$.
Denote by $\bar\nu_E^P$ and $\bar\eta_E^P$ the laws of $(\bF(e):e\in E)$ and $(\bGG(e):e\in E)$ on $\cS^E$.
Let $(\bar B^e:e\in E)$ be a family of coalescing Brownian motions, $\bar B^e$ running backwards in time from $x(e)$ at time $s(e)$.
Thus $\bar B^e=(\bar B_t^e:0\le t\le s(e))$ and for all $e,e'\in E$, $\bar B^e$ and $\bar B^{e'}$ are independent
until (time running backwards) they collide, at which time they coalesce.
Let $(\bar W^e:e\in E)$ be a family of coalescing Brownian motions,
with $\bar W^e$ running forwards in time from $x(e)$ at time $s(e)$.
Denote by $\bar\nu_E$ and $\bar\eta_E$ the laws on $\cS^E$ of the families of random sets $(\{(t,\bar B_t^e):0\le t\le s(e)\}:e\in E)$
and $(\{(t,\bar W_{t\wedge T}^e):t\ge s(e)\}:e\in E)$.
\begin{theorem}\label{LOCAL}
We have $\bar\nu_E^P\to\bar\nu_E$ and $\bar\eta_E^P\to\bar\eta_E$ weakly on $\cS^E$ as $\d\to0$.
\end{theorem}

\bg
Thus, for small $\d$, we can construct on a common probability space, the cluster $K_N$ and
backwards and forwards coalescing Brownian motions, such that the union
of fingers in $\tilde K_N$ starting from points
$s(e)+ix(e)\sqrt{\d^*}$, $e \in E$ is, with probability close to $1$, close in Hausdorff metric to the set
$\bigcup_{e \in E} \{t+i \bar B^e_t \sqrt{\d^*}: 0 \le t \le s(e) \}$, and hence the union of fingers in $K_N$,
starting from points $\exp(s(e)+ix(e)\sqrt{\d^*})$, $e \in E$ looks approximately like the set
$\bigcup_{e \in E} \{\exp (t+i\bar B^e_t \sqrt{\d^*}): 0 \le t \le s(e) \}$.
Similarly, the union of gaps in $K_N$, starting from points $\exp(s(e)+ix(e)\sqrt{\d^*})$, $e \in E$ looks
approximately like the set $\bigcup_{e \in E} \{\exp (t+i\bar W^e_{t\wedge T}\sqrt{\d^*}): t \geq s(e) \}$.
A simulation of $\bigcup_{e \in E} \{\exp (t+i\bar B^e_t \sqrt{\d^*}): 0 \le t \le s(e) \}$ and
$\bigcup_{e \in E} \{\exp (t+i\bar W^e_t \sqrt{\d^*}): s(e) \le t \le T \}$ is shown in Figure \ref{locallimitfig}.
\eb

Theorems \ref{LONG} and \ref{LOCAL} are obvious corollaries of
Theorem \ref{FLOWFING}, which identifies also the limiting joint law of fingers and gaps.

\begin{figure}[ht]
    \subfigure[An approximation of a finite set of fingers and gaps in $K_N$, with $N=\lfloor \rho T\rfloor$, when $T=1$ and $\d^*=0.05$.]
      {\label{globallimitfig}
      \epsfig{file=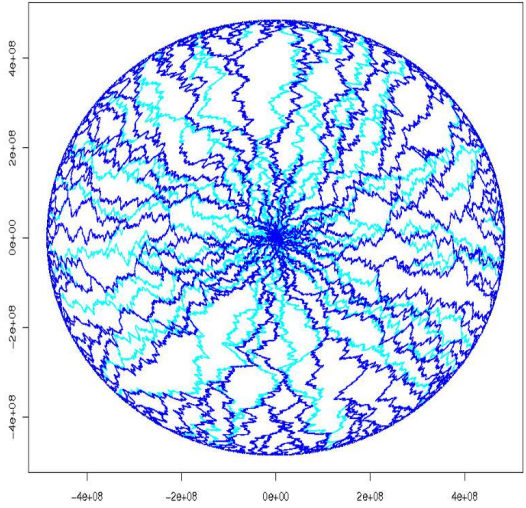,width=7.5cm}}
    \hfill
    \subfigure[An approximation of a finite set of fingers and gaps in $K_N$, with $N=\lfloor c^{-1} T\rfloor$, when $T=1$ and $\d^*=0.01$.]
      {\label{locallimitfig}
      \epsfig{file=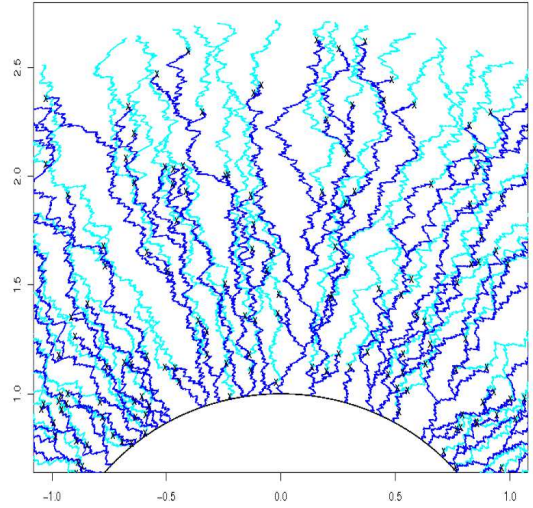,width=7.5cm}}
  \caption{\textsl{Geometric illustration of Theorems \ref{LONG} and \ref{LOCAL}, where fingers are denoted in dark blue, and gaps in light blue.}}
  \label{limitfig}
\end{figure}

\section{Some basic estimates}\label{SBE}
We derive in this section some estimates for quantities associated to the basic particle $P$.
In some special cases one could use instead an explicit calculation. By proving general estimates we are able to
demonstrate some universality for the small-particle limit.
Recall that $K=K_0\cup P$ and $D=(\C\cup\{\infty\})\sm K$, \bg with $K$ compact and locally connected and $D$ simply connected in $\C\cup\{\infty\}$.
\eg
The following assumptions are in force throughout this section
\bb
\begin{equation}\label{D13}
\d\in(0,1/3] \q\text{and}\q P\sse\{z\in\C:|z-1|\le\d\}\q\text{and}\q  1+\d \in P \q\text{and}\q P=\{\bar z : z \in P\}.
\end{equation}
\eb

Consider the map $\psi(z)=\bar z^{-1}$ on $\C\cup\{\infty\}$ by reflection in the unit circle $S$.
Set $\hat P=\psi(P)$ and $\hat D=\psi(D)$, $\hat D_0=\psi(D_0)$.
Define also $P^*=P\cup I\cup\hat P$, where $I$ is the set of limit points of $P$ in $S$, and set $D^*=(\C\cup\{\infty\})\sm P^*$.
By the Riemann mapping theorem, there is a conformal map $\hat G:\hat D\to\hat D_0$
and a constant $c\in\R$ such that $\hat G(z)=e^cz+O(|z|^2)$ as $|z|\to0$, and $\hat G$ and $c$ are unique.
Moreover $\hat G$ extends to a conformal map $G^*:D^*\to(\C\cup\{\infty\})\sm J$ for some interval $J\sse S$,
with $G^*\circ\psi=\psi\circ G^*$ on $D^*$.
Write $G$ for the restriction of $G^*$ to $D$.
Then $G$ is a conformal map $D\to D_0$ and $G(z)=e^{-c}z+O(1)$ as $|z|\to\infty$,
The constant $c$ is the logarithmic capacity $\cp(K)$.
The well known fact that $c$ is positive will emerge in the course of the proof of Proposition \ref{GZE}.

Note that $D^*$ is simply connected and $G^*(z)/z\not=0$ for all $z\in D^*$.
So we may choose a branch of the logarithm
so that $\log(G^*(z)/z)$ is continuous on $D^*$ with limit $c$ at $0$ and then,
for some constant $C(K)<\infty$, we have
$$
\left|\log\left(\frac{\hat G(z)}{z}\right)-c\right|\le{C(K)}|z|,\q z\in\hat D
$$
and so
$$
\left|\log\left(\frac{G(z)}{z}\right)+c\right|\le\frac{C(K)}{|z|},\q z\in D.
$$
In fact the following stronger estimate holds.

\begin{proposition}\label{GZE}
There is an absolute constant $C<\infty$ such that
$$
\left|\log\left(\frac{G(z)}{z}\right)+\cp(K)\right|\le\frac{C\cp(K)}{|z-1|},\q |z-1|>2\d,\q z\in D.
$$
\end{proposition}
\begin{proof}
Set $H(z)=u(z)+iv(z)=\log(G^*(z)/z)$.
Then $H$ is bounded and holomorphic on $D^*$ and $H(z)\to-c$ as $|z|\to\infty$.
Fix $z\in\C$ and let $B$ be a complex Brownian motion starting from $z$.
Suppose that $z\in D$ and consider the stopping time
$$
T=\inf\{t\ge0:B_t\not\in D\}.
$$
Then $T<\infty$ and $|B_T|\ge1$ almost surely, and $|B_T|>1$ with positive probability.
Also $u(B_t)\to-\log|B_T|$ as $t\ua T$ almost surely.
Hence, by optional stopping,
$$
u(z)=-\E(\log|B_T|)<0.
$$
Set $r=\d/(2-\d)$ and define \bb $P_1^*=\{z\in\C:|z-1|\le r|z+1|\}$. Then set
$$
D_1^*=(\C\cup\{\infty\})\sm P_1^*,\q P_1=P_1^*\cap D_0,\q D_1=D_1^*\cap D_0, \q K_1=K_0 \cup P_1.
$$\eb
Then $P^*\sse P_1^*\sse \{z\in\C:|z-1|\le\d/(1-\d)\}$.
The boundary of $D_1$ consists of two circular arcs, one contained in $S$, where $u=0$, the other contained in $P_1$, which we denote by $A$.
The  normalized conformal map $G_1:D_1\to D_0$ can be obtained as $\phi^{-1}\circ g_1\circ\phi$, where $\phi$ takes
$D_0$ to the upper half-plane by $\phi(z)=i(z-1)/(z+1)$ and $g_1(z)=(z+r^2/z)/(1-r^2)$.
Hence we obtain $G_1^*(z)=z(\g z-1)/(z-\g)$ for $z\in D_1^*$, where $\g=(1-r^2)/(1+r^2)$, and
$G_1(A)=\{e^{i\th}:|\th|<\th_0\}$, where $\th_0=\cos^{-1}\g$.
Set $F_1=G_1^{-1}$. Then $u\circ F_1$ is bounded and harmonic on $D_0$.
Suppose now that $z\in D_0$ and consider the stopping time
$$
T_0=\inf\{t\ge0:B_t\not\in D_0\}.
$$
Then $T_0<\infty$ almost surely and, by optional stopping,
$$
u(F_1(z))=\E(u(F_1(B_{T_0})))=\frac1{2\pi}\int_{|\th|\le\th_0}u(F_1(e^{i\th}))\re\left(\frac{z+e^{i\th}}{z-e^{i\th}}\right)d\th.
$$
On letting $|z|\to\infty$ we obtain
$$
c=-\frac1{2\pi}\int_{|\th|\le\th_0}u(F_1(e^{i\th}))d\th>0
$$
so
$$
u(F_1(z))+c=\frac1{2\pi}\int_{|\th|\le\th_0}u(F_1(e^{i\th}))\re\left(\frac{2e^{i\th}}{z-e^{i\th}}\right)d\th.
$$
Hence, for $z\in D_1$,
$$
|u(z)+c|\le\frac{2c}{\dist(G_1(z),G_1(A))}.
$$
By an elementary calculation, we have $|(G_1^*)'(z)-\g|\le6\g/7$ whenever $|z-1|\ge7\d/4$ and $\d\in(0,1/3]$.
Set $A'=\{z\in\C:|z-1|=7\d/4\}$. Then $\dist(G_1(z),G_1(A))\ge\dist(G_1(z),G_1(A'))$ whenever $|z-1|\ge7\d/4$.
By the mean value theorem, there is an absolute constant \bb $C_1<\infty$ \eb such that
$$
\dist(G_1(z),G_1(A'))\ge |z-1|/C_1,\q |z-1|\ge2\d.
$$
Hence
$$
|u(z)+c|\le 2C_1c/|z-1|,\q |z-1|\ge2\d,\q z\in D
$$
and the same estimate extends to $D^*$ by reflection.

Then, by a standard estimate for harmonic functions \bb (differentiate the Poisson kernel)\eb,
\begin{equation}\label{NUV}
\left|\nabla v(z)\right|=\left|\nabla u(z)\right|\le8C_1c/|z-1|^2,\q |z-1|\ge2\d,\q z\in D
\end{equation}
and so
\bb
$$
|v(z)|\le \int_0^\infty |\nabla v(z+s(z-1)||z-1|ds \le 8C_1c/|z-1|,\q |z-1|\ge2\d,\q z\in D,
$$
giving the required bound.
\eb
\end{proof}

\bg
\begin{corollary}\label{GZEE}
We have $\d^2 /6 \le \cp(K)\le 3\d^2/4$.
\end{corollary}
\eg
\begin{proof}
\bg
We use notation from the preceding proof.
\eg
By uniqueness, we have $G_1=G^\dagger\circ G$, where $G^\dagger$ is the normalized
conformal map $G(D_1)\to D_0$. Hence
\bb
$$
\cp(K) \le \cp(K)+\cp(G(K_1\sm K))=\cp(K_1)=\log\left(\frac{1+r^2}{1-r^2}\right) \le \frac{3\d^2}4.
$$
Also, since $1+\d \in P$, $G=G^\ddag \circ G_2$, where $G_2$ is the normalized slit map
$D_2 = D_0 \sm (1, 1+\d] \to D_0$ referred to in Section \ref{illus} and $G^\ddag$ is the normalized conformal map $G_2(D) \to D_0$. Let $K_2=K_0 \cup (1, 1+\d]$. Then
$$
\cp(K) = \cp(K_2) + \cp(G_2(K \sm K_2)) \geq - \log \left ( 1 - \frac{\d^2}{(2+\d)^2}\right ) \geq \frac{\d^2}6.
$$
\eb
\end{proof}

We shall do most of the analysis in logarithmic coordinates.
Set $\tilde D=\{z\in\C:e^z\in D\}$ and $\tilde D_0=\{z\in\C:\re(z)>0\}$.
There are unique conformal maps $\tilde G:\tilde D\to\tilde D_0$ and $\tilde F:\tilde D_0\to\tilde D$
such that $\tilde G(z)-z+c\to0$ and $\tilde F(z)-z-c\to0$ as $\re(z)\to\infty$.
Then $\tilde F$ and $\tilde G$ are $2\pi i$-periodic and $\tilde F=\tilde G^{-1}$.
Also $G\circ\exp=\exp\circ\,\tilde G$ and $F\circ\exp=\exp\circ\tilde F$.
Proposition \ref{GZE} and (\ref{NUV}) provide the following estimates for $\tilde G(z)$
\begin{equation}\label{TGE}
|\tilde G(z)-z+c|\le\frac{Cc}{|e^z-1|},\q |\tilde G'(z)-1|\le\frac{Cc|e^z|}{|e^z-1|^2},\q |e^z-1|\ge2\d,\q z\in\tilde D.
\end{equation}

We introduce some further functions associated to $\tilde G$ and $\tilde F$.
\bb Recall the definitions of $I$ and $J$ from the start of this section. Since $P$ is symmetric, \eb we can write $I=\{e^{i\th}:|\th|\le p\}$ and $J=\{e^{i\th}:|\th|\le q\}$ for some $p\in[0,\pi)$ and $q\in(0,\pi)$.
Then there exist unique non-decreasing right-continuous functions $g^+$ and $f^+$ on $\R$ such that the functions $\th\mapsto g^+(\th)-\th$
and $\th\mapsto f^+(\th)-\th$ are $2\pi$-periodic and such that
\begin{equation}\label{DGF}
g^+(\th)=\begin{cases}\pm q,& \pm\th\in(0,p]\\\im(\tilde G(i\th)),& |\th|\in(p,\pi]\end{cases},\q
f^+(\th)=\begin{cases}0,& |\th|\in[0,q)\\\im(\tilde F(i\th)),& |\th|\in(q,\pi]\end{cases}.
\end{equation}
Here we have used the continuous extensions of $\tilde G$ and $\tilde F$ to certain intervals of the imaginary axis.
Define, for $\th\in\R$
$$
g_0(\th)=g^+(\th)-\th
$$
and, for $x\in(0,1]$ such that $x+i\th\in\tilde D$, define
$$
g_x(\th)=\im(\tilde G(x+i\th))-\th.
$$

\begin{proposition}\label{TGES}
There is an absolute constant $C<\infty$ such that, for $\a=C\d$ and $|\th|\le\pi$,
$$
|g_0(\th)|\le\frac{\a^2}{|\th|\vee\a}
$$
and the same estimate holds for $|g_x(\th)|$ when $x\in(0,1]$ and $x+i\th\in\tilde D$.
Moreover $C$ may be chosen so that
$$
\d^3/C\le\frac1{2\pi}\int_0^{2\pi}g_0(\th)^2d\th\le C\d^3,\q
\frac1{2\pi}\int_0^{2\pi}|g_0(\th)g_0(\th+a)|d\th\le \frac{C\d^4}a\log\left(\frac1\d\right)
$$
whenever $a\in[\d,\pi]$.
\end{proposition}
\begin{proof}
The first estimate follows from the first estimate in (\ref{TGE}), using the \bb non-decreasing property of $g^+$ \eb\bg and the maximum principle \eg
to deal with the case where $|e^{x+i\th}-1|<2\d$, and using $\cp(K)\le 3\d^2/4$. This leads directly to the
upper bound in the second estimate and the third estimate.

For the lower bound, note that
$$
\frac1{2\pi}\int_0^{2\pi}g_0(\th)^2d\th\ge\frac1\pi\int_0^q(q-\th)^2d\th=\frac{q^3}{3\pi}
$$
and $q=\pi\PP_\infty(B_T\in P)$.
\bg We give an argument which uses neither the symmetry assumption $P=\{\bar z:z\in P\}$ nor the assumption
$1+\d\in P$ and instead assumes only that $|z-1|=\d$ for some $z\in P$. This will be useful in Lemma \ref{NOT}. \eg
Denote by $P^{(2)}$ the union of $P$ with its reflection in the line $\ell$ joining $z$ and $0$.
Denote by $w$ the image of $1$ under this reflection and by $A$ the shorter arc in the unit circle joining $w$ and $1$.
Then, since $P$ is connected, we have
\begin{align}\notag
2\PP_\infty(B_T\in P)&\ge\PP_\infty(B\text{ hits $P^{(2)}$ before $K_0$})\\\notag
&\ge\PP_\infty(B\text{ hits $K_0$ in $A$})\vee\PP_\infty(B\text{ hits $\ell$ before $K_0$})\\\label{EHM}
&\ge(|w-1|\vee(|z|-1))/(2\pi)\ge\d/(4\pi)
\end{align}
which gives the claimed lower bound.
\end{proof}

\def\comm{
Following does not seem needed here, but probably will be in a later section.
Our analysis rests on the following estimate for the conformal map $G:D\to D_0$.
The argument is of a standard type but we have taken care obtain
an explicit dependence of the right hand side on $|z|$ and $\d$ which we
shall find useful.
\def\j{
\footnote{The following stronger stronger estimate also holds
$$
\left|\log\left(\frac{G(z)}{z}\right)+\cp(K)\frac{z+1}{z-1}\right|\le\frac{C\d^3}{|z|},\q |z-1|\ge2\d .
$$
Lawler \cite{L}[Proposition 3.55], modulo inversion in the unit circle, has $C\d^3|z|^2/|z-1|^2$ on the right,
with the condition $|1-1/z|\ge C\d$,
which is insufficiently informative for us for large $|z|$. Must check!}
}
\begin{proposition}\label{GEST}
Assume that $\d\le1/16$. Then $\cp(K)\le16\d^2$ and there is a universal constant $C<\infty$ such that
$$
\left|\log\left(\frac{G(z)}{z}\right)+\cp(K)\frac{z+1}{z-1}\right|\le\frac{C\d^3|z|}{(|z|-1)^2},\q |z|>1+C\d.
$$
\end{proposition}
\begin{proof}
Consider the image domain under inversion $\hat D=\{z\in\C:1/z\in D\}$ and write $\hat G$ for the unique
conformal map taking $\hat D$ to the open unit disc such that $\hat G(0)=0$ and $\hat G'(0)>0$.
Then $G(z)=1/\hat G(1/z)$ and $\cp(K)=\log\hat G'(0)$.
There is a constant $C(K)<\infty$ such that $|\hat G(z)-e^{\cp(K)}z|\le C(K)|z|^2$ for all $z\in\hat D$.
From this it follows that $G(z)/z\to e^{-\cp(K)}$ as $|z|\to\infty$, that we may choose a branch of the logarithm
so that $\log(G(z)/z)$ is continuous on $D$ with limit $-\cp(K)$ at $\infty$, and then that,
for some constant $C(K)<\infty$, we have
$$
\left|\log\left(\frac{G(z)}{z}\right)+\cp(K)\right|\le\frac{C(K)}{|z|},\q z\in D.
$$

Consider the holomorphic functions
$$
h(z)=\log\left(\frac{G(z)}{z}\right),\q
\tilde h(z)=\log\left(\frac{G(z)}{z}\right)+\cp(K)\frac{z+1}{z-1},\q z\in D.
$$
Write $h(z)=u(z)+iv(z)$ and $\tilde h(z)=\tilde u(z)+i\tilde v(z)$.
Note, by the preceding estimate, that $u$ is bounded on $D$, and $u(z)\to-\cp(K)$ and $\tilde v(z)\to0$ as $|z|\to\infty$.
Fix $z\in D$ and let $B$ be a complex Brownian motion starting from $z$.
Consider the stopping time $\t=\inf\{t\ge0:B_t\not\in D\}$.
Then $\t<\infty$ and $u(B_t)\to-\log|B_\t|$ as $t\ua\t$, almost surely.
Hence, by optional stopping, $u(z)=-\E_z(\log|B_\t|)$. In particular $-\d\le-\log(1+\d)\le u(z)\le0$.

Consider the conformal map $\phi(z)=i(z-1)/(z+1)$ taking $D_0$ to the upper half plane $H_0$.
Note that $\phi'(z)=2i/(z+1)^2$ so, if $|z-1|\le\d\le1$, then $|z+1|\ge1$, so $|\phi'(z)|\le2$.
Then, by the mean value theorem, since $\phi(1)=0$, the image of $P$ under $\phi$ is contained in $B_r(0)$,
the closed disc of radius $r=2\d$ about $0$. Set $H^*=H_0\sm B_r(0)$ and $D^*=\phi^{-1}(H^*)$, then $D^*\sse D$.
Denote by $A$ the smaller circular arc in the boundary of $D^*$ and parametrize $A$ by
$e(\th)=\phi^{-1}(re^{i\th}),\th\in[0,\pi]$. Now $u$ is bounded and harmonic, and continuous
on the closure of $D^*$, vanishing on the larger circular arc in the boundary. Hence, for $z\in D^*$,
$$
u(z)=\int_0^\pi u(e(\th))p_{D^*}(z,d\th),
$$
where $p_{D^*}(z,d\th)=p_{D^*}(z,\th)d\th$ is the hitting distribution of $B$ on $\partial D^*$, restricted to $A$.
Let $|z|\to\infty$ to obtain
$$
\cp(K)=-\int_0^\pi u(e(\th))p_{D^*}(\infty,d\th).
$$
It is straightforward to see, for example using the conformal map $G^*$ below, that $p_{D^*}(\infty,A)\le16\d$.
We recover the well known bounds $0\le\cp(K)\le16\d^2$.

Define
$$
\D(z)=\sup_\th\left|\frac{p_{D^*}(z,\th)}{p_{D^*}(\infty,\th)}-\frac{|z|^2-1}{|z-1|^2}\right|
$$
and note that
$$
\tilde u(z)=u(z)+\cp(K)\frac{|z|^2-1}{|z-1|^2}
=\int_0^\pi\left(\frac{p_{D^*}(z,\th)}{p_{D^*}(\infty,\th)}-\frac{|z|^2-1}{|z-1|^2}\right)u(e(\th))p_{D^*}(\infty,d\th).
$$
We shall show that, for some universal constant $C<\infty$, we have
\begin{equation}\label{DELT}
\D(z)\le C\d|z|/(|z|-1)^2, \q\text{whenever }|z|\ge1+C\d.
\end{equation}
On using this estimate in the preceding integral, we deduce that
$$
|\tilde u(z)|\le C\cp(K)\d|z|/(|z|-1)^2\le C\d^3|z|/(|z|-1)^2.
$$
Hence, following a standard argument,
$$
|\nabla\tilde v(z)|=|\nabla\tilde u(z)|\le C\d^3(|z|+1)/(|z|-1)^3,
$$
and then
$$
|\tilde v(z)|\le\int_1^\infty|z||\nabla\tilde v(tz)|dt\le\int_1^\infty
C\d^3|z|\frac{t|z|+1}{(|tz|-1)^3}dt=\frac{C\d^3|z|}{(|z|-1)^2}.
$$

It remains to show the estimate (\ref{DELT}). Consider the conformal map $g^*:H^*\to H_0$ given by
$g^*(z)=z+r^2/z$ and the corresponding map $G^*:D^*\to D_0$ given by
$$
G^*(z)=\phi^{-1}\circ g^*\circ\phi(z)=z\frac{1+a}{1-a},\q a=\frac{z+1}{z-1}r^2.
$$
Then $G^*(e(\pi))=\phi^{-1}(2r)=1+4r/(i-2r)$, so $|G^*(e(\pi))-1|\le 4r=8\d$.
Fix $z\in D^*$, note that $|a|\le r$ and recall that $\d\le1/16$.
Fix $\th\in[0,\pi]$ and set $w=G^*(e(\th))$ and $z'=G^*(z)$.
Then $|w-1|\le8\d$ and $|z-z'|=2|z||a/(1-a)|\le2r|z|/(1-r)\le8\d|z|$.
Also $|z'-w|\ge|z-1|-|z-z'|-|w-1|\ge|z-1|-8\d|z-1|-16\d\ge|z-1|/4$, provided $|z|\ge1+64\d$.

By the invariance of harmonic measure under conformal maps
$$
\frac{p_{D^*}(z,\th)}{p_{D^*}(\infty,\th)}=\frac{p_{D_0}(z',w)}{p_{D_0}(\infty,w)}=\re\left(\frac{z'+w}{z'-w}\right).
$$
So, for $|z|\ge1+64\d$,
$$
\left|\frac{p_{D^*}(z,\th)}{p_{D^*}(\infty,\th)}-\frac{|z|^2-1}{|z-1|^2}\right|\le\left|\frac{z'+w}{z'-w}-\frac{z+1}{z-1}\right|
=2\frac{|(w-1)z+(z-z')|}{|z'-w||z-1|}\le\frac{C\d|z|}{|z-1|^2}.
$$
Since $\th$ was arbitrary, this gives the claimed bound on $\D(z)$.
\end{proof}

\br
\begin{proposition}\label{GDL}
There is a universal constant $C<\infty$ such that $\d^{-3}/C \le \rho(P) \le C\d^{-3}$
and $\l(P)\le C\d^{1/4}$.
\end{proposition}
\begin{proof}
It is shown in Lawler \cite{L} that there exists some universal constant $c<\infty$ such that if
$c\d\le x\le1-c\d$, then
$$
|\tilde{g}(x)| \le \frac{3 \cp(P \cup K_0)}{2\pi|1 - e^{2 \pi ix}|} = \frac{3\cp(P \cup K_0)}{4\pi \sin(\pi x)},
$$
where $\cp(P \cup K_0) = - \log G'(\infty)$ is the logarithmic capacity of $P \cup K_0$. If $K_1 \sse K_2$,
then $\cp(K_1) \le \cp(K_2)$ (see, for example, \cite{L}),
and hence $\cp(P \cup K_0) \le\cp(L\cup K_0) \le c'\d^2$
for some constant $c' > 0$, by the estimate \eqref{LCE}.

Now suppose $x \in (-c\d, c\d)$. Then $g(-c\d) \le g(x) \le g(c\d)$ and hence, for some $c''>0$,
$$
|\tilde{g}(x)|
\le|\tilde{g}(c\d)|\vee|\tilde{g}(-c\d)|+2c\d
\le\frac{3c'\d^2}{4\pi\sin(\pi c\d)}+2c\d
\le c''\d.
$$
Hence, we have
\begin{align*}
\rho(g)^{-1}&=\int_0^1|\tilde{g}(x)|^2 dx
\le2c\d(c''\d)^2+\int_{c\d}^{1-c\d}\frac{9c'^2\d^4}{16\pi^2\sin^2(\pi x)}dx\\
&=2cc''^2\d^3+\frac{9c'^2\d^4}{8\pi^3}\cot(c\pi\d)=O(\d^3).
\end{align*}

To show the other side of the inequality, recall the estimate \eqref{FR3}, which implies
that $\rho(g)\le8\|\tilde{g}\|^{-3}$. So it is enough to show that there exists
some $x \in [0,1)$ with $|\tilde{g}(x)| \geq c'''\d$ for some $c'''>0$.
Define
$$
\tau_A = \inf \{t>0: B_t \in A \},
$$
where $B_t$ is a Brownian motion starting from infinity. Since Brownian motion is invariant under conformal transformations,
$$
g^+(0)-g^-(0)
=\PP(\t_P\le\t_{K_0}).
$$
Since $P$ has diameter $\d$, there exist two points, $w, z \in \overline{P}$,
with $w \in \overline{P} \cap K_0$ and $|w-z| \geq \d/2$. Since $P$ is connected there is a path,
$l \sse \overline{P}$ joining $w$ and $z$. Let $r$ be the line joining 0 to $z$. Reflect $l$ in $r$ and call the reflection $l'$ and the reflection of point $w$, $w'$. Now
$$
2 \mathbb{P}(\tau_l \le \tau_{K_0})\geq \mathbb{P}(\tau_{l \cup l'} \le \tau_{K_0} )
$$
and
$$
\mathbb{P}(\tau_{l \cup l'} \le \tau_{K_0} ) \geq \mathbb{P}(\tau_r \le \tau_{K_0}) \geq 8c'''(|z|-1),
$$
by the limit \eqref{DTHN}, but also
$$
\mathbb{P}(\tau_{l \cup l'} \le \tau_{K_0}) \geq \mathbb{P}(B_{\tau_{K_0}} \in \text{arc } (w-w')) \geq 8c'''|w-w'|
$$
for some constant $c'''>0$. Either $(|z|-1)$ or $|w-w'| \geq \d/4$ which finishes the argument.

Finally, to show that $\l(g)\le C\d^{1/4}$ observe that for $\d$ sufficiently small,
if $a \in [\d^{1/4}, 1- \d^{1/4}]$, then at most one of $x$ and $x+a$ lie in
$\cup_{n \in \mathbb{Z}}[n-c\d, n+c\d]$. Hence
$$
\rho\int_0^1|\tilde{g}(x+a)\tilde{g}(x)|dx
\le2\rho c''\d\int_{c\d}^{1-c\d}\frac{3c'\d^2}{4\pi\sin(\pi x)}dx
=\frac{3\rho c'c''\d^3}{\pi^2}\log\tan(\pi c\d)
=O(\d).
$$
Hence there exists some $C>0$ such that, for all $a \in [C\d^{1/4}, 1- C\d^{1/4}]$, we have
$$
\rho \int_0^1|\tilde{g}(x+a)\tilde{g}(x)|dx < C\d^{1/4}.
$$
\end{proof}
\er
}

\section{Fluid limit analysis for random conformal maps}\label{FLA}
Define conformal maps $\tilde F_n$ and $\tilde\Phi_n$ on $\tilde D_0$ by
$$
\tilde F_n(z)=\tilde F(z-i\Th_n)+i\Th_n,\q \tilde\Phi_n=\tilde F_1\circ\dots\circ\tilde F_n
$$
where $(\Th_n:n\in\N)$ is the sequence of independent uniformly distributed random variables
specified in the Introduction. Write $\tilde\G_n$ for the inverse map $\tilde\Phi_n^{-1}:\tilde D_n\to\tilde D_0$.
It will be convenient to use the filtration $(\cF_n:n\ge0)$ given by $\cF_n=\s(\Th_1,\dots,\Th_n)$.
Recall that we write $c$ for the logarithmic capacity $\cp(K)$.
Assumption (\ref{D13}) remains in force in this section.

For $\ve\in[2\d,1]$ and $m\in\N$, denote by $\O(m,\ve)$ the event defined by the following conditions:
for all $z\in\tilde D_0$ and all $n\le m$, we have
$$
|\tilde\Phi_n(z)-z-cn|<\ve\q\text{whenever}\q\re(z)\ge5\ve
$$
and
$$
z\in\tilde D_n\q\text{and}\q|\tilde\G_n(z)-z+cn|<\ve\q\text{whenever}\q\re(z)\ge cn+4\ve.
$$
\bg Note the round brackets -- this is not the same event as $\O[m,\ve]$, defined above.
We shall use the following estimate in the case where $m=\lfloor\d^{-6}\rfloor$ and $\ve=\d^{2/3}\log(1/\d)$ when,
using the bound $c\le 3\d^2/4$ from Corollary \ref{GZEE}, it implies that $\O(m,\ve)$ has high probability as $\d\to0$.
The proof is based on a fluid limit approximation for each Markov process $(\tilde\G_n(z):n\ge0)$, optimized using explicit
martingale estimates. Local uniformity in $z$ is achieved by combining the estimates for individual starting points
with an application of Kolmogorov's H\"older criterion.
\eg

\begin{proposition}\label{FLUD}
There is an absolute constant $C<\infty$ such that, for all $\ve\in[2\d,1]$ and all $m\in\N$,
$$
\PP(\O\sm\O(m,\ve))\le C\bb (m + \ve^{-2}) \eb e^{-\ve^3/(Cc)}.
$$
\end{proposition}
\begin{proof}
It will suffice to consider the case where $\ve^3\ge c$.
Set $M=\lceil cm/(2\ve)\rceil$.
Fix $k\in\{1,\dots,M\}$ and set $R=2(k+1)\ve$.
Consider the vertical line $\ell_R=\{z\in\C:\re(z)=R\}$.
Write $N$ for the largest integer such that $cN\le R-2\ve$.
Consider the stopping time
$$
T=T_R=\inf\{n\ge0:z\not\in\tilde D_n\text{ or }\re(\tilde\G_n(z))\le R-cn-\ve
\text{ for some }z\in\ell_R\}\wedge N.
$$
Note that
$$
\re(\tilde\G_{T-1}(z))>\ve \bb > \eb \d>\log(1+\d)
$$
so $z\in\tilde D_{T_R}$ for all $z\in\ell_R$.
Consider the events
$$
\O_R=\left\{\sup_{n\le T_R,\,z\in\ell_R}|\tilde\G_n(z)-z+cn|<\ve\right\},\q
\O_0(m,\ve)=\bigcap_{k=1}^M\O_{2(k+1)\ve}.
$$
We shall show that there is an absolute constant $C<\infty$ such that
\begin{equation}\label{ESS}
\bb \PP(\O\sm\O_R)\le C\ve^{-6/5}e^{-\ve^3/(Cc)} \eb
\end{equation}
from which it follows that
\bb
$$
\PP(\O\sm\O_0(m,\ve))\le C ( cm/ \ve +1) \ve^{-6/5} e^{-\ve^3/(Cc)} \le C(m+\ve^{-2})e^{-\ve^3/(Cc)}.
$$
\eb

Note that, on $\O_R$, we have $|\tilde\G_{T_R}(z)-z+cT_R|<\ve$ for all $z\in\ell_R$,
which forces $T_R=N$ and so $z\in\tilde D_n$ whenever $\re(z)\ge R$ and $cn\le R-2\ve$. Then, since
$\tilde\G_n(z)-z+cn$ is a bounded holomorphic function on $\tilde D_n$, we have on $\O_R$
$$
\sup_{cn\le R-2\ve,\,\re(z)\ge R}|\tilde\G_n(z)-z+cn|=
\sup_{cn\le R-2\ve,\,z\in\ell_R}|\tilde\G_n(z)-z+cn|<\ve.
$$
For \bb $n\le m$\eb, we can choose $k$ so that $R-4\ve\le cn\le R-2\ve$.
Then, if $\re(z)\ge cn+4\ve$, then $\re(z)\ge R$, so on $\O_0(m,\ve)$ we have $z\in\tilde D_n$ and
$|\tilde\G_n(z)-z+cn|<\ve$. Moreover, on $\O_0(m,\ve)$,
the image $\tilde\G_n(\ell_R)$ lies to the left of $\ell_{5\ve}$,
and hence, if $\re(w)\ge5\ve$, we have $w=\tilde\G_n(z)$ for some $\re(z)\ge R$ so that
$$
|\tilde\Phi_n(w)-w-cn|=|z-\tilde\G_n(z)-cn|<\ve.
$$
We have shown that $\O_0(m,\ve)\sse\O(m,\ve)$, which implies the claimed estimate.

It remains to prove (\ref{ESS}).
The function $\tilde G_0(z)=\tilde G(z)-z$ is holomorphic, bounded and $2\pi i$-periodic on $\tilde D$
with $\tilde G_0(z)\to-c$ as $\re(z)\to\infty$.
Hence
$$
\frac1{2\pi}\int_0^{2\pi}\tilde G_0(z-i\th)d\th=-c,\q \re(z)>\d.
$$
\bb
Let $q(r)=r\wedge r^2$. Then
$$
|\tilde G_0(z)+c|\le\frac{C_1c}{\re(z)-\d},\q |\tilde G_0'(z)|\le\frac{2C_1c}{q(\re(z)-\d)},\q\re(z)\ge2\d,
$$
where $C_1$ is the absolute constant in (\ref{TGE}). \eb
Set
$$
M_n(z)=\tilde\G_n(z)-z+cn,\q z\in\tilde D_n.
$$
Then
$$
M_{n+1}(z)-M_n(z)=\tilde G_0(\tilde\G_n(z)-i\Th_{n+1})+c.
$$
So $(M_n(z))_{n\le T}$ is a martingale for all \bb $z\in \ell_R$.
For $z\in\ell_R$ and $n\le T-1$,
$$
|M_{n+1}(z)-M_n(z)|\le\frac{C_1c}{(\re(\tilde\G_n(z))-\d)}\le\frac{C_1c}{(R-cn-\ve-\d)}
$$
and
$$
\sum_{n=0}^{N-1}\frac{C_1^2c^2}{(R-cn-\ve-\d)^2}
\le\int_0^{R-2\ve}\frac{C_1^2cds}{(R-s-\ve-\d)^2}=\frac{C_1^2c}{\ve-\d}\le \frac{2C_1^2c}{\ve}.
$$
So, by the Azuma-Hoeffding inequality, for all $z\in\ell_R$,
\begin{equation}\label{MZ}
\PP\left(\sup_{n\le T}|M_n(z)|\ge\ve/2\right)\le2e^{-\ve^3/(16C_1^2c)}.
\end{equation}
\eb
Fix $z,z'\in\ell_R$, define $\tilde M_n=M_n(z)-M_n(z')$ and set
$$
f(n)=\E\left(\sup_{k\le T\wedge n}|\tilde M_k|^2\right).
$$
Note that $|\tilde\G_n(z)-\tilde\G_n(z')|\le|z-z'|+|\tilde M_n|$ so, for $n\le T-1$,
$$
|\tilde M_{n+1}-\tilde M_n|=|\tilde G_0(\tilde\G_n(z)-i\Th_{n+1})-\tilde G_0(\tilde\G_n(z')-i\Th_{n+1})|
\le\frac{2C_1c(|z-z'|+|\tilde M_n|)}{q(R-cn-\ve-\d)}.
$$
Then, by Doob's $L^2$-inequality,
$$
f(n)\le4\E\left(|\tilde M_{T\wedge n}|^2\right)\le4\sum_{k=0}^{n-1}\E(|\tilde M_{k+1}-\tilde M_k|^21_{\{k\le T\}})
\le32C_1^2c^2\sum_{k=0}^{n-1}\frac{|z-z'|^2+f(k)}{q(R-ck-\ve-\d)^2}
$$
so, by a Gronwall-type argument,
$$
\E\left(\sup_{n\le T}|\tilde M_n|^2\right)=f(N)\le|z-z'|^2\left(\exp\int_{\ve-\d}^\infty\frac{32C_1^2cds}{q(s)^2}-1\right).
$$
Now, for $r\in(0,1]$,
$$
\int_r^\infty\frac{ds}{q(s)^2}=\frac13\left(2+\frac1{r^3}\right)
$$
and $\ve/2\le\ve-\d\le1$ and $\ve^3\ge c$. So we deduce the existence of an absolute constant \bb$16 C_1^2 < C_2<\infty$ such that \eb $f(N)\le C_2c|z-z'|^2/\ve^3$.
Hence by Kolmogorov's lemma, $C_2$ may be chosen so that, for some random variable \bb $M$, \eb with $\E(M^2)\le C_2c/\ve^3$, we have
\bb
$$
\sup_{k\le T}|M_k(z)-M_k(z')|\le M|z-z'|^{1/3}
$$
\eb
for all $z,z'\in\ell_R$. So, by Chebyshev's inequality, for any $L\in\N$,
\bb
$$
\PP\left(\sup_{k\le T}|M_k(z)-M_k(z')|\ge\ve/2\text{ for some $z,z'\in\ell_R$ with $|z-z'|\le \pi /L$}\right)\le(\pi/L)^{2/3}4C_2c/\ve^5.
$$
\eb
On combining this with (\ref{MZ}), we obtain
\bb
$$
\PP(\O\sm\O_R)\le L e^{-\ve^3/(C_2c)}+(\pi/L)^{2/3}4C_2c/\ve^5
$$
\eb
from which (\ref{ESS}) follows on optimizing over $L$.
\end{proof}

We note two consequences of the event $\O(m,\ve)$. First we deduce an estimate for the normalized conformal maps $\Phi_n$.
On $\O(m,\ve)$, for $n\le m$ and $|z|=e^{5\ve}$, we have
$$
|\log(e^{-cn}\Phi_n(z))-\log z|<\ve
$$
and so
$$
|e^{-cn}\Phi_n(z)-z|<\ve e^{6\ve}.
$$
The last estimate then holds whenever $|z|\ge e^{5\ve}$ by the maximum principle.

Second, we show that on the event $\O(m,\ve)$, for $n\le m$ and $R\le cn$, there is no disc of radius $\bb 56\ve \eb$  with centre on the line $\ell_R=\{z\in\C:\re(z)=R\}$
which is disjoint from $\tilde K_n$.
Since the sets $\tilde K_n$ are increasing in $n$, we may assume that $R>c(n-1)$.
Fix $y\in\R$ and set $w=6\ve+iy$. Note that \bb$|\tilde\Phi_n(w)-(R+iy)|< \ve+|6\ve+cn-R| < 7\ve  + c < 8 \ve$. Here we have used $c \le 3\d^2/4 \le \d/4 < \ve$\eb.
By Cauchy's integral formula
$$
\tilde\Phi_n'(w)=1+\frac1{2\pi i}\int_{|z-w|=\ve}\frac{\tilde\Phi_n(z)-z-cn}{(z-w)^2}dz
$$
so $|\tilde\Phi_n'(w)|\le2$. Then, by Koebe's $1/4$ theorem,
$$
d(\tilde\Phi_n(w),\partial\tilde D_n)\le 4|\tilde\Phi_n'(w)|d(w,\partial\tilde D_0)\le 48\ve.
$$
and so
\begin{equation}\label{NBH}
d(R+iy,\partial\tilde D_n)\le \bb 56\ve. \eb
\end{equation}

\section{Harmonic measure and the location of particles}\label{HMLP}
In this section we obtain an estimate on the location of the particles $P_{n+1}=\Phi_n(e^{i\Th_{n+1}}P)$ in the plane.
From the preceding section, we know that $\Phi_n((1+\ve)e^{i\th})$ is close to $(1+\ve)e^{cn+i\th}$ with high probability, when $\ve$ is suitably
large in relation to the particle radius $\d$.
This must break down as $\ve\to0$, at least when particles are attached at a single point, since the map $\th\mapsto\Phi_n(e^{i\th})$
parametrizes the whole cluster boundary by harmonic measure.
Nevertheless, we shall show that the approximation breaks down only on a set of very small harmonic measure, and in fact
the whole of each particle $P_{n+1}$ is close to $e^{cn+i\Th_{n+1}}$, in a sense made precise below.
Throughout this section, we assume that condition \eqref{D13} holds and we make also the following \bg {\em non-degenerate contact} \eg condition
\begin{equation}\label{D14}
P\sse\{z\in\C:\re(z)>1\}.
\end{equation}

\begin{lemma}\label{NOT}
There is an absolute constant $C<\infty$ with the following properties.
Let $D^*$ be any simply connected neighbourhood of $\infty$ in $D_0$ and set $K^*=\C\sm D^*$.
Denote by $\mu$ the harmonic measure from $\infty$ in $D^*$ of $K^*\sm K_0$
and by $N$ the number of connected components of $K^*\sm K_0$.
Then
\begin{equation}\label{SN}
\PP(K^*\cap K_1\not= K_0)\le CN\sqrt\mu.
\end{equation}
Assume further that $16\pi\mu\le\d$.
Then
$$
\PP(K^*\cap K_\infty\not= K_0)\le CN\sqrt\mu/\d.
$$
\end{lemma}
\begin{proof}
By the estimate (\ref{EHM}), each of the $N$ connected components of $K^*\sm K_0$ is
contained in a disc of radius $8\pi\mu$ with centre on the unit circle.
The non-degenerate contact assumption
then allows us to choose \bb $C_1<\infty$ \eb such that $P_1$ intersects that component only if $e^{i\Th_1}$ lies in
a concentric disc of radius $C_1\sqrt\mu$. The estimate (\ref{SN}) follows.

Consider a complex Brownian motion $B$ with $B_0$ uniformly distributed on the circle of radius $2$
centred at $0$, and independent of $\Th_1$. Set $r=1+\d/2$ and note that $K^*\sse rK_0$. Set
$$
T(K)=\inf\{t\ge0:B_t\in K\}.
$$
Note that, since $\Th_1$ is uniformly distributed on $[0,2\pi)$, the events $\{T(rK_0)\le T(K_1)\}$ and
$\{T(K^*)<T(K_0)\}$ are independent. We use the estimate (\ref{EHM}) and our assumption that $1+\d\in P$ to obtain
$$
\PP(T(rK_0)\le T(K_1))\le 1-\d/C_1.
$$
Note that, since $T(rK_0)\le T(K^*)$, we have
$$
\{T(K^*)<T(K_1)\}\sse\{T(rK_0)\le T(K_1)\}\cap\{T(K^*)<T(K_0)\}.
$$
Hence
$$
\PP(T(K^*)<T(K_1))\le(1-\d/C)\PP(T(K^*)<T(K_0))=(1-\d/C_1)\mu.
$$

Set
$$
H_n=\PP(T(K^*)<T(K_n)|\cF_n)
$$
and note that $H_0=\mu$.
By conformal invariance of Brownian motion, we have
$$
H_n=\PP(T(K_n^*)<T(K_0)|\cF_n)
$$
where $K^*_n=\G_n(K^*\sm K_n)\cup K_0$, and moreover
$$
\E(H_{n+1}|\cF_n)=\PP(T(K^*_n)<T(K'_1)|\cF_n)
$$
where $K'_1$ is an independent copy of $K_1$.
Since $K^*_n\sse rK_0$, the argument of the preceding paragraph applies to show that
$\E(H_{n+1}|\cF_n)\le(1-\d/C_1)H_n$. Hence $\E(H_n)\le(1-\d/C_1)^n\mu$ for all $n$.

On the event $\{K^*\cap K_n=K_0\}$, the set $K_n^*\sm K_0$ has $N$ connected components, and its harmonic
measure from $\infty$ in \bb $\C\sm K^*_n$ \eb is $H_n$. Define $P_1'=\G_n(P_{n+1})$. Then $P'_1$ has the same distribution as $P_1$
and is independent of $\cF_n$. So the argument leading to (\ref{SN}) applies to give
$$
\PP(K^*\cap P_{n+1}\not=\es|\cF_n)=\PP(K_n^*\cap P'_1\not=\es|\cF_n)\le C_1N\sqrt{H_n}.
$$
Hence
\begin{align*}
\PP(K^*\cap K_\infty\not= K_0)
&\le\sum_{n=0}^\infty\PP(\{K^*\cap K_n=K_0\}\cap\{K^*\cap P_{n+1}\not=\es\})\\
&\le\sum_{n=0}^\infty C_1N\E(\sqrt{H_n})\le C_1^2N\sqrt\mu/\d.
\end{align*}
\end{proof}

Write $\tilde P$ for the connected component of $\tilde K\sm\tilde K_0$ near $0$. Set
$$
\tilde P_n=\tilde\Phi_{n-1}(\tilde P+i\Th_n),\q \tilde A_n=\tilde\Phi_{n-1}(i\Th_n).
$$
Then $\tilde P_n$ is a component of the $2\pi i$-periodic set $\tilde K_n\sm\tilde K_{n-1}$
and it is attached to $\tilde K_{n-1}$ at $\tilde A_n$.
\bg
For the next result, we shall use a further assumption on the particle $P$ which allows us
to prove that none of the sets $\tilde P_n$ contain a certain size of fjord, even though they
have been distorted by the maps $\tilde\Phi_{n-1}$. The useful form of this assumption is expressed
in terms of harmonic measure. After stating this, we will give a geometrically more obvious sufficient
condition.
We assume the following {\em harmonic measure condition}.
\begin{align}
&\text{For all sequences $(z_1,w_1,z_2,w_2)$ of points in $\partial P$, listed anticlockwise, and for any}\notag\\
&\text{interval $I$ of $\partial D_0$, if for $i=1$ and $i=2$ at least the $3/4$ of the harmonic measure}\notag\\
&\text{on $\partial D_0$ from $w_i$ is carried on $I$, then for either $i=1$ or $i=2$ at least $1/4$ of the}
\label{HMCO}\\
&\text{harmonic measure on $\partial D_0$ from $z_i$ is carried on $I$.}\notag
\end{align}

This condition is implied by the following property of the image $\phi(P)$, where $\phi$ is the conformal map from $D_0$ to the upper half-plane $H_0$,
as in footnote 4. For $z=x+iy$ and $z'=x'+iy'$ in $H_0$, write $S(z,z')$ for the smallest closed square in $H_0$ containing
all the points $x-y,x+y,x'-y',x'+y'$. Then the preceding harmonic measure condition is implied by the following {\em square condition}.
For all $z,z'\in\partial(\phi(P))$,
at least one of the boundary arcs of $\partial(\phi(P))$ from $z$ to $z'$ is contained in $S(z,z')$. To see this, suppose $I\sse\R$ is an interval
which carries at least $3/4$ of the harmonic measure on $\R$ starting from $z$, then $(x-y,x+y)\sse I$. Hence, if the same is true for $z'$, then
$S(z,z')\cap\R\sse I$. Then, for any point $w\in S(z,z')$, $I$ carries at least $1/4$ of the harmonic measure on $\R$ starting from $w$.
We have used here the fact that the harmonic measure on $\R$ starting from $i$ places equal mass on the intervals $(-\infty,-1),(-1,0),(0,1),(1,\infty)$.
It is easy to check the square condition for $P=(1,1+\d]$ and $P=\{|z-1+\d/2|=\d/2\}$, when $\phi(P)$ is also a slit or a disc.
\eg

Consider for $\nu\in[0,\infty)$ the event
$$
\O(m,\ve,\nu)=\{\re(z)>c(n\wedge m)-\ve\nu\text{ for all $z\in\tilde P_{n+1}$ and all $n\ge0$}\} \cap \O(m,\ve).
$$
\bg
In conjunction with Proposition \ref{FLUD}, the following estimate implies that, when $m=\lfloor\d^{-6}\rfloor$ and $\ve=\d^{2/3}\log(1/\d)$
and $\nu=(\log(1/\d))^2$, the event $\O(m,\ve,\nu)$ has high probability as $\d\to0$.
\eg

\begin{proposition}\label{NOTLATE}
There exists an absolute constant $C<\infty$ such that, for all \bb$\epsilon \in [2\d,1]$ and \eb $\nu\in[0,\infty)$,
$$
\PP(\O(m,\ve)\sm\O(m,\ve,\nu))\le Cm \bb (m+\d^{-1}) \eb e^{-\nu/C}.
$$
\end{proposition}
\begin{proof}
We use the following Beurling estimate. There is an absolute constant $A\in[1,\infty)$ with the following property.
For any $\eta\in(0,1]$ and any connected set $K$ in $\C$ joining the circles of radius $\eta$ and $1$ about $0$,
the probability that a complex Brownian motion, starting from $0$,
leaves the unit disc without hitting $K$ is no greater than $A\sqrt{\eta}$.

Fix $n\le m$ with $\ve\nu\le cn$. Condition on $\cF_n$ and on $\O(n,\ve)$.
For all $z\in\C$ with $0\le\re(z)\le cn$, there exists
$w\in\tilde K_n$ such that \bb $|z-w|\le56\ve$.
Set $\b=56A^2e^2$ \eb  and $\nu_0=\lfloor \nu/(2\b)\rfloor$.
We assume without loss that $\nu_0\ge6$.
Define $R(k)=cn-\b\ve k$ and note that $R(2\nu_0)\ge0$.
Fix $k\in\{0,1,\dots,\nu_0-1\}$ and $z\in\ell_{R(k)}$ and consider a complex Brownian motion $B$ starting from $z$.
By the Beurling estimate, $B$ hits $\ell_{R(k+1)}$ without hitting $\tilde K_n$ with probability no greater than $\bb A\sqrt{56/\b}=e^{-1} \eb$.
\bg Then, by the strong Markov property, for all $z\in\ell_{R(0)}$, almost surely on $\O(n,\ve)$,
\begin{equation}\label{HMK}
\PP_z(B\text{ hits }\ell_{R(\nu_0)}\text{ before }\tilde K_n|\cF_n)\le e^{-\nu_0}\le e^{-\nu/(2\b)+1}.
\end{equation} \eg
There exists a family of disjoint open intervals $((\th_j,\th_j'):j=1,\dots,N_n)$ in $\R/(2\pi\Z)$
such that, for $w_j=\tilde\Phi_n(i\th_j)$ and $w_j'=\tilde\Phi_n(i\th_j')$,
we have $\re(w_j)=\re(w_j')=R(\nu_0)$ and $\bigcup_j(w_j,w_j')+2\pi i\Z$ disconnects $\tilde D_n\cap\ell_{R(2\nu_0)}$ from $\infty$ in $\tilde D_n$.
We choose the unique such family minimizing $\sum_j|w_j-w_j'|$.
Then $w_j\in\tilde P_{k(j)}$ for some $k(j)\le n$ for all $j$.
We shall show that the integers $k(1),\dots,k(N_n)$ must all be distinct, so $N_n\le n$.

\bg
Suppose $k(j)=k(j')=k_0+1$ for some distinct $j$ and $j'$. Then there exist $\a<\b<\a'<\b'<\a+2\pi$ such that,
for $z=\tilde\Phi_{k_0+1}(i\a)$, $z'=\tilde\Phi_{k_0+1}(i\a')$, $w=\tilde\Phi_{k_0+1}(i\b)$ and $w'=\tilde\Phi_{k_0+1}(i\b')$,
we have $z,z',w,w'\in\partial\tilde P_{k_0+1}$ and $\re(z)=\re(z')=R(2\nu_0)$ and $\re(w)=\re(w')=R(\nu_0)$. Then, since we are on $\O(m,\ve)$, we must have $c(k_0+1)+4\ve\ge R(\nu_0)$, so $ck_0\ge R(\nu_0)-4\ve-c\ge R(\nu_0+1)$. Hence there exists an interval $I$ of $\partial\tilde D_{k_0}$ with endpoints $p,p'$ in $\ell_{R(3\nu_0/2)}$ such that $z,z'$ are separated from $\partial\tilde D_{k_0}\sm I$ by $I\cup[p,p']$.
By a variation of the Beurling and strong Markov argument above, all but $e^{-\nu_0/2}$ of the harmonic measure
on $\partial\tilde D_{k_0}$ starting from $z$ is carried on $I$, and the same is true for $z'$.
Then, by conformal invariance of harmonic measure, all but $e^{-\nu_0/2+1}<1/4$ of the harmonic measure
on $i\R$ starting from $\tilde F_{k_0+1}(i\a)$ is carried on $\tilde\G_{k_0}(I)$, and the same is true for $\a'$.
So, by our harmonic measure condition, either more than $1/4$ of the harmonic measure on $i\R$ starting from $\tilde F_{k_0+1}(i\b)$ is carried
on $\tilde\G_{k_0}(I)$, or the analogous statement holds for $\b'$.
But, by the Beurling and strong Markov argument again, no more than $e^{-\nu_0/2+1}<1/4$ of the harmonic measure
on $\partial\tilde D_{k_0}$ starting from $w$ is carried on $I$, and the same is true for $w'$.
So, by conformal invariance, no more than $1/4$ of the harmonic measure on $i\R$ starting from $\tilde F_{k_0+1}(i\b)$ is carried
on $\tilde\G_{k_0}(I)$, and the same is true for $\b'$, a contradiction.
\er

Each path $(\tilde\Phi_n(i\th):\th\in(\th_j,\th_j'))$, together with the line segment $[w_j,w_j']$,
forms the boundary of a connected subset of $\tilde D_n$.
Denote by $S_n$ the union of these subsets.
Define $K^*_n=\{e^{\tilde\G_n(z)}:z\in S_n\}\cup K_0$ and $D^*_n=(\C\cup\{\infty\})\sm K^*_n$.
Then $D^*_n$ is a simply connected neighbourhood of $\infty$ in $D_0$, the set $K^*_n\sm K_0$ has $N_n$ connected components and, by (\ref{HMK}), the harmonic measure from $\infty$
of $K^*_n\sm K_0$ in $D_n^*$ is no greater than $e^{-\nu/(2\b)+1}$.
So,  \bb on $\O(n,\ve)$, \eb
$$
\PP((e^{i\Th_{n+1}}P)\cap K_n^*\not=\es|\cF_n)\le C_1N_ne^{-\nu/(4\b)+1/2},
$$
\bb where $C_1$ is the absolute constant from Lemma \ref{NOT}. \eb
But, if $e^{i\Th_{n+1}}P$ does not meet $K_n^*$, then $\re(z)>cn-\nu\ve$ for all $z\in\tilde P_{n+1}$.
Of course this inequality holds also in the case where $cn<\nu\ve$.

It remains to deal with the case where $n\ge m+1$. We may assume that $\nu\ge2\b\log(16\pi e/\d)$ or the
estimate is trivial. Then, for $\mu=e^{-\nu/(2\b)+1}$, we have $16\pi\mu\le\d$.
So we can apply Lemma \ref{NOT} with $K^*=K^*_m$
to obtain, \bb on $\O(m,\ve)$, \eb
$$
\PP(\re(z)\le cm-\ve\nu\text{ for some }z\in\tilde P_{n+1}\text{ and some }n\ge m|\cF_m)\le C_1N_me^{-\nu/(4\b)+1/2}/\d.
$$
The estimates we have obtained combine to prove the proposition.
\end{proof}

\bb
\begin{rem}
An analogous result to Proposition \ref{NOTLATE} can be obtained by bounding the contribution to the length of the cluster boundary made by each particle. This extends the class of allowable basic particles beyond that specified by \eqref{HMCO}, but at the expense of a weaker bound on the probability. 

Suppose that \eqref{D13} and \eqref{D14} hold and that, in addition, $\partial P$ is rectifiable, with length $L$, and is given by $\beta:[0,L] \to \partial P$, where the parametrization is by arc length. We assume further that $\beta$ is piecewise differentiable in such a way that there exist $C(\d)\in (0,\infty)$, $k(\d)\in\N$ and  $0=a_0<a_1<\cdots<a_{k(\d)}=L$ such that $r_i:(a_i, a_{i+1}) \to (1,1+\d)$ given by $r_i(t)=|\beta(t)|$ is differentiable with $|r_i'(t)|>C(\d)$ on $(a_i, a_{i+1})$ for all $i=0, \dots, k(\d)-1$. Set $r(\d)=k(\d)/C(\d)$. Let $\tilde \ell_{n+1}$ be the contribution to the length of the boundary of $\partial \tilde K_{n+1}$ that comes from particle $\tilde P_{n+1}$. Then
$$
\tilde \ell_{n+1} = \int_0^L \frac{|\Phi_n'(\beta(t)e^{i \Theta_{n+1}})|}{|\Phi_n(\beta(t)e^{i \Theta_{n+1}})|} dt.
$$
So, by a similar argument to that in the proof of Theorem 4 of \cite{RZ},
\begin{eqnarray*}
\E(\tilde \ell_{n+1} | K_n ) &=& \sum_{i=0}^{k(\d)-1} \int_0^{2 \pi} \int_{a_i}^{a_{i+1}} \frac{|\Phi_n'(r_i(t)e^{i \theta})|}{|\Phi_n(r_i(t)e^{i \theta})|} dt d\theta \\
&\le& r(\d) \int_0^{2 \pi} \int_1^{1+\d} \frac{|\Phi_n'(re^{i \theta})|}{|\Phi_n(re^{i \theta})|} dr d\theta \\
&\le& C_1 r(\d) (cn\d)^{1/2},
\end{eqnarray*}
for some absolute constant $C_1<\infty$. Therefore, if $N_n$ is defined as in the proof of Proposition \ref{NOTLATE}, for all $\zeta>0$, $\PP(N_n > \zeta ) \le \sum_{j=1}^n \PP(\tilde \ell_j > \zeta \nu \ve n^{-1}) \le C_1 r(\d) \d^{3/2}n^{5/2}/(\zeta \nu \ve)$. Hence, there exists some absolute constant $C<\infty$ such that
$$
\PP(\O(m,\ve)\sm\O(m,\ve,\nu))\le C\zeta (m+\d^{-1}) e^{-\nu/C} + Cr(\d)\d^{3/2}m^{7/2}/(\zeta \nu \ve ).
$$
On optimizing over $\zeta$, it can be shown that there exists another absolute constant $C<\infty$ such that
$$
\PP(\O(m,\ve)\sm\O(m,\ve,\nu))\le C r(\d)^{1/2} m^{7/4}\d^{3/4}\nu^{-1/2}\ve^{-1/2}(m+\d^{-1})^{1/2}e^{-\nu/C}. 
$$
\end{rem}
\eb

Define for $z\in\tilde D_0$
$$
N(z)=\inf\{n\ge0:z\not\in\tilde D_n\}.
$$
Denote by $\O(m,\ve,\nu,\eta)$ the subset of $\O(m,\ve,\nu)$ defined by the following condition:
for all $z\in\tilde D_0\cap\tilde K_\infty$ with $N(z)\le m$ and all $n\le N(z)-1$, we have
$$
|\im(\tilde\G_n(z)-z)|<\ve+2\eta.
$$
\bg
In conjunction with Propositions \ref{FLUD} and \ref{NOTLATE}, the following estimate implies that, when $m=\lfloor\d^{-6}\rfloor$ and $\ve=\d^{2/3}\log(1/\d)$
and $\nu=(\log(1/\d))^2$ and $\eta=\d^{2/3}(\log(1/\d))^6$, the event $\O(m,\ve,\nu,\eta)$ has high probability as $\d\to0$.
\eg

\begin{proposition}\label{OMNE}
There is an absolute constant $C<\infty$ such that, for all \bb $\ve \in [2 \d, 1/6]$, $\nu \in [0, \infty)$ and \eb $\eta\in(0,\infty)$,
$$
\PP\left(\O(m,\ve,\nu)\sm\O(m,\ve,\nu,\eta)\right)\le\frac{Cm}{\eta}\exp\left\{-\frac\eta{C\d}+\frac{C\nu\ve\d}c\left(1+\log\left(1/\d\right)\right)\right\}.
$$
\end{proposition}
\begin{proof}
Fix $z\in\tilde D_0\cap\tilde K_\infty$ with $N(z)\le m$.
Write $N_0(z)$ for \bb the maximum of 0 and \eb the largest integer such that $cN_0(z)\le\re(z)-4\ve$.
Write $N_1$ for the smallest integer such that $cN_1\ge (\nu+4)\ve$.
Then, on $\O(m,\ve,\nu)$, we have $N(z)-1\le N_0(z)+N_1$ and, since $\O(m,\ve,\nu)\sse\O(m,\ve)$, we have also
$|\im(\tilde\G_k(z)-z)|<\ve$ for all $k\le N_0(z)$.

We showed in Proposition \ref{TGES} that, for some absolute constant \bb $C_1<\infty$, for $\a=C_1\d$ \eb and for all $z\in\tilde D_0$ with $\re(z)\le1$,
$$
\im(\bb \tilde G \eb(z))\le g^*(\im(z))
$$
where $g^*(\th)=\th+g_0^*(\th)$ and $g_0^*$ is the $2\pi$-periodic function given by
$$
g_0^*(\th)=\frac{\a^2}{|\th|\vee\a},\q \th\in(-\pi,\pi].
$$
Then, for $N_0(z)\le n\le N(z)-1$,
$$
\im(\tilde\G_n(z))\le Y_n^{(n_0,y_0)}
$$
where $n_0=N_0(z)$, $y_0=\im(\tilde\G_{n_0}(z))$ and where, recursively for $n\ge n_0$, $Y_n=Y_n^{(n_0,y_0)}$ is defined by
$$
Y_{n_0}=y_0,\q Y_{n+1}=g^*(Y_n-\Th_{n+1})+\Th_{n+1}=g_0^*(Y_n-\Th_{n+1})+Y_n.
$$
Note that \bb $\re(\G_n(z)) \le \re(\G_{n_0}(z))\le |\re(\G_{n_0})-\re(z)+cn_0|+\re(z)-cn_0  < 5\ve+c <1$. Hence, \eb  $g_0^*$ is non-negative and $g^*$ is non-decreasing, so $Y_n^{(n_0,y_0)}$ is non-decreasing in $n$ and $y_0$.

Set $M=\lceil 2\pi/\eta\rceil$ and $h=2\pi/M$ so that $h\le\eta$.
Consider the set of time-space starting points
$$
E=\{(n_0,jh):n_0\in\{0,1,\dots,m\},j\in\{0,1,\dots,M-1\}\}
$$
and the event
$$
\O_0=\{Y_{n_0+N_1}^{(n_0,jh)}\le jh+\eta\text{ for all }(n_0,jh)\in E\}.
$$
Note that $|E|\le Cm/\eta$, so
$$
\PP(\O\sm\O_0)\le Cm\PP(Y_{N_1}>\eta)/\eta
$$
where $Y=Y^{(0,0)}$.
Now
$$
\E(e^{Y_1/\a})=\frac1{2\pi}\int_{-\pi}^\pi e^{g_0^*(\th)/\a}d\th=1+\frac{\a(e-1)}\pi+\frac1\pi\int_\a^\pi(e^{\a/\th}-1)d\th
\le\exp\{(\a e/\pi)(1+\log(\pi/\a))\}
$$
so
$$
\PP(Y_{N_1}>\eta)\le\exp\{-\eta/\a+(N_1\a e/\pi)(1+\log(\pi/\a))\}.
$$

Choose $j\in\{1,\dots,M\}$ so that $(j-1)h\le y_0\le jh$. Then, for $n\le N(z)-1$, on $\O_0$,
$$
\im(\tilde\G_n(z))\le Y_n^{(n_0,jh)}\le jh+\eta\le\im(\tilde\G_{n_0}(z))+2\eta\le\im(z)+\ve+2\eta.
$$
A similar argument allows us to bound the downward variation of $\im(\tilde\G_n(z))$ up to $N(z)-1$.
Hence $\PP(\O(m,\ve,\nu)\sm\O(m,\ve,\nu,\eta)\le2\PP(\O\sm\O_0)$ which gives the claimed estimate.
\end{proof}

\bg
\begin{theorem}\label{BALLG}
Assume that the basic particle $P$ satisfies conditions (\ref{D13}),(\ref{D14}) and (\ref{HMCO}).
Consider for $\ve_0\in(0,1]$ and $m\in\N$ the event $\O[m,\ve_0]$ specified by the following conditions:
for all $n\le m$ and all $n'\ge m+1$,
$$
|z-e^{cn+i\Th_n}|\le\ve_0 e^{cn}\q\text{for all $z\in P_n$}
$$
and
$$
\dist(w,K_n)\le\ve_0 e^{cn}\q\text{whenever $|w|\le e^{cn}$}
$$
and
$$
|z|\ge(1-\ve_0)e^{cm}\q\text{for all $z\in P_{n'}$}.
$$
Assume that $\ve_0=\d^{2/3}(\log(1/\d))^8$ and $m=\lfloor\d^{-6}\rfloor$. Then
$\PP(\O[m,\ve_0])\to1$ uniformly in $P$ as $\d\to0$.
\end{theorem}
\begin{proof}
Set $\ve=\d^{2/3}\log(1/\d)$ and $\nu=(\log(1/\d))^2$ and $\eta=\d^{2/3}(\log(1/\d))^6$.
We have shown that the event $\O(m,\ve,\nu,\eta)$ has high probability as $\d\to0$.
We complete the proof by showing that, for $\d$ sufficiently small, the defining conditions
for $\O[m,\ve_0]$ are all satisfied on $\O(m,\ve,\nu,\eta)$.

Fix $n\le m$ and $z\in\tilde P_n$.
On $\O(m,\ve)$ we have $\re(z)<cn+4\ve$ and, restricting to $\O(m,\ve,\nu)$, we have also $\re(z)>c(n-1)-\nu\ve$.
Restricting further to $\O(m,\ve,\nu,\eta)$, we have $|\im(\tilde\G_{n-1}(z)-z)|<\ve+2\eta$.
But $\tilde\G_{n-1}(z)\in\tilde P+2\pi i\Th_n$, so $|\tilde\G_{n-1}(z)-2\pi i\Th_n|\le\d$.
Hence, on $\O(m,\ve,\nu,\eta)$, we have (since $\nu\ge4$)
$$
|e^z-e^{cn+2\pi i\Th_n}|\le e^{cn}(e^{\nu\ve+c}-1)+e^{cn+\nu\ve+c}(\ve+2\eta+\d).
$$
We can choose $\d$ sufficiently small that
$$
(e^{\nu\ve+c}-1)+e^{\nu\ve+c}(\ve+2\eta+\d)\le\ve_0.
$$
Then on $\O(m,\ve,\nu,\eta)$ we have, for all $z\in P_n$,
$$
|z-e^{cn+2\pi i\Th_n}|\le \ve_0e^{cn}.
$$

Next, using (\ref{NBH}), for $0\le\re(w)\le cn$, on $\O(m,\ve)$, there exists $z\in\tilde K_n$ with $|z-w|\le 56\ve$.
Then $e^z\in K_n$ and
$$
|e^z-e^w|\le 56\ve e^{cn+4\ve}.
$$
We can choose $\d$ sufficiently small that
$$
56\ve e^{4\ve}\le\ve_0.
$$
Then $\dist(w,K_n)\le\ve_0e^{cn}$ whenever $|w|\le e^{cn}$ and $n\le m$.

Finally, for $n\ge m+1$ and $z\in\tilde P_n$, on $\O(m,\ve,\nu)$, we have $\re(z)>cm-\nu\ve$.
Hence $|w|>e^{cm-\nu\ve}\ge(1-\ve_0)e^{cm}$ for all $w\in P_n$.
\end{proof}
\eg

\section{Weak convergence of the localized disturbance flow to the coalescing Brownian flow}\label{WC}
We review in this section the main results of \cite{NT1}.
Denote by $\bar\cD$ the set of all pairs $f=\{f^-,f^+\}$,
where $f^+$ is a right-continuous, non-decreasing function on $\R$
and where $f^-$ is the left-continuous modification of $f^+$.
Denote by $\cD$ the subset of those $f\in\bar\cD$ such that $x\mapsto f^+(x)-x$ is periodic of period $2\pi$.
Write $\id$ for the identity function $\id(x)=x$ and, for $f\in\bar\cD$, write $f_0^\pm$ for the periodic functions $f^\pm-\id$.
Denote by $\cD^*$ the subset of $\cD$ where $f_0$ is not identically zero but has zero mean
$$
\frac1{2\pi}\int_0^{2\pi}f_0(x)dx=0.
$$
Here and below, we drop the $\pm$ where the quantity computed takes the same value for both versions.
Fix $f\in\cD^*$ and define $\rho=\rho(f)\in(0,\infty)$ by
$$
\frac\rho{2\pi}\int_0^{2\pi}f_0(x)^2dx=1.
$$
Let $(\Th_n:n\in\Z)$ be a sequence of independent random variables, all uniformly distributed on $[0,2\pi)$.
Define for each non-empty bounded interval $I\sse\R$ a pair of random functions $\Phi_I=\{\Phi_I^-,\Phi_I^+\}$ by
$$
\Phi_I^\pm=f_{\Th_n}^\pm\circ\dots\circ f_{\Th_m}^\pm
$$
where $f_\th^\pm(x)=f^\pm(x-\th)+\th$ and where $m$ and $n$ are, respectively, the smallest and largest integers in the rescaled interval $\rho I$.
If $\rho I\cap\Z=\es$ then $\Phi_I=\id$.
Write $I=I_1\oplus I_2$ if $I_1, I_2$ are disjoint intervals with $\sup I_1=\inf I_2$ and $I=I_1\cup I_2$.
Note that the family $\Phi=(\Phi_I:I\sse\R)$ has the following {\em flow property}
\begin{equation}\label{FP}
\Phi_{I_2}^\pm\circ \Phi_{I_1}^\pm=\Phi_I^\pm,\q\text{whenever }I=I_1\oplus I_2.
\end{equation}
Moreover (see \cite{NT1}), almost surely, for all $I$, $\Phi_I^-$ is the left-continuous modification of $\Phi_I^+$,
so $\Phi_I=\{\Phi_I^-,\Phi_I^+\}\in\cD$.
We call $\Phi$ the {\em disturbance flow with disturbance $f$.}
For $\ve\in(0,1]$, we make the diffusive rescaling
$$
\Phi_I^{\ve,\pm}(x)=\ve^{-1}\Phi_{\ve^2I}^\pm(\ve x),\q x\in\R
$$
and call $(\Phi^\ve_I:I\sse\R)$ the {\em $\ve$-scale disturbance flow with disturbance $f$.}

In order to formulate a weak convergence result about these disturbance flows, we introduce metrics on $\cD$ and $\bar\cD$
and then we define certain metric spaces which will serve as state-spaces for $\Phi$ and $\Phi^\ve$.
First, define for $f,g\in\cD$
$$
d_\cD(f,g)=\inf\{\ve\ge0:f^+(x)\le g^+(x+\ve)+\ve\text{ and }g^+(x)\le f^+(x+\ve)+\ve\text{ for all $x\in\R$}\}.
$$
For $f,g\in\bar\cD$, define
$$
d_{\bar\cD}(f,g)=\sum_{n=1}^\infty2^{-n}(d_n(f,g)\wedge1)
$$
where
$$
d_n(f,g)=\inf\{\ve\ge0:f^+(x)\le g^+(x+\ve)+\ve\text{ and }g^+(x)\le f^+(x+\ve)+\ve\text{ for all $x\in[-n,n-\ve]$}\}.
$$
Then $d_\cD$ is a metric on $\cD$ and the metric space $(\cD,d_\cD)$ is complete.
In fact $(\cD,d_\cD)$ is isometric to the set of periodic contractions on $\R$ with period $2\pi$, with
supremum metric, by drawing new axes for the graph of $f\in\cD$ at a rotation of $\pi/4$.
Also, $d_{\bar\cD}$ is a metric on $\bar\cD$ and the metric space $(\bar\cD,d_{\bar\cD})$ is complete.
See \cite{NT1}.

Consider now a family $\phi=(\phi_I:I\sse\R)$, where $\phi_I\in\cD$ and $I$ ranges over all non-empty bounded intervals.
Say that $\phi$ is a {\em weak flow} if,
\begin{equation}\label{WF}
\phi_{I_2}^-\circ\phi_{I_1}^-\le\phi_I^-\le\phi_I^+\le\phi_{I_2}^+\circ\phi_{I_1}^+,
\q\text{whenever }I=I_1\oplus I_2.
\end{equation}
Say that $\phi$ is {\em cadlag} if, for all $t\in\R$,
$$
d_\cD(\phi_{(s,t)},\id)\to0\q\text{as}\q s\uparrow t\q\q\text{and}\q\q
d_\cD(\phi_{(t,u)},\id)\to0\q\text{as}\q u\downarrow t.
$$
We write $D^\circ(\R,\cD)$ for the set of all cadlag weak flows.
For the disturbance flow $\Phi$, almost surely, for all $t\in\R$,
for all sufficiently small $\ve>0$, we have $\Phi_{(t-\ve,t)}=\Phi_{(t,t+\ve)}=\id$.
So $\Phi$ takes values in $D^\circ(\R,\cD)$.
Define similarly $D^\circ(\R,\bar\cD)$ and note that $\Phi^\ve$ takes values in $D^\circ(\R,\bar\cD)$.

Fix $\phi\in D^\circ(\R,\cD)$ and suppose that $\phi_{\{t\}}=\id$ for all $t\in\R$.
Then $\phi_{(s,t)}=\phi_{(s,t]}=\phi_{[s,t)}=\phi_{[s,t]}$ for all $s,t\in\R$ with $s<t$.
Denote all these functions by $\phi_{ts}$ and set $\phi_{tt}=\id$ for all $t\in\R$.
The map
$$
(s,t)\mapsto\phi_{ts}:\{(s,t)\in\R^2:s\le t\}\to\cD
$$
is then continuous.
We write $C^\circ(\R,\cD)$ for the set of such {\em continuous weak flows} $\phi$
and we write $C^\circ(\R,\bar\cD)$ for the analogous subset in $D^\circ(\R,\bar\cD)$.

We can and do make $D^\circ(\R,\cD)$ and $D^\circ(\R,\bar\cD)$ into complete
separable metric spaces by the choice of Skorokhod-type metrics, both denoted $d_D$.
The metrics $d_D$ have the following two further properties.
The associated Borel $\s$-algebras coincide with those generated by the evaluation maps $\phi\mapsto\phi_I^+(x)$ as $x$ ranges over $\R$ and
$I$ ranges over bounded intervals in $\R$.
Moreover, for any sequence $(\phi^n:n\in\N)$ in $D^\circ(\R,\cD)$ and any $\phi\in C^\circ(\R,\cD)$, we have $d_D(\phi^n,\phi)\to0$ if and only if
$d_\cD(\phi_I^n,\phi_I)\to0$ uniformly over subintervals $I$ of compact sets in $\R$.
In particular, $C^\circ(\R,\cD)$ is closed in $D^\circ(\R,\cD)$.
Analogous statements hold in the non-periodic case.
However, the flow property (\ref{FP}) is not preserved under limits in $d_D$.
We refer to \cite{NT1} for the specification of $d_D$.

The disturbance flow $\Phi$ with disturbance $f$ is then a $D^\circ(\R,\cD)$-valued random variable,
and the law of $\Phi$ is a Borel probability measure on $D^\circ(\R,\cD)$, which we denote by $\mu_A^f$.
The $\ve$-scale disturbance flow $\Phi^\ve$ is a $D^\circ(\R,\bar\cD)$-valued random variable,
so the law of $\Phi^\ve$ is a Borel probability measure on $D^\circ(\R,\bar\cD)$, which we denote by $\mu_A^{f,\ve}$.

For $e=(s,x)\in\R^2$ and $\phi\in D^\circ(\R,\cD)$, the maps
$$
t\mapsto\phi_{(s,t]}^-(x):[s,\infty)\to\R,\q
t\mapsto\phi_{(s,t]}^+(x):[s,\infty)\to\R
$$
are cadlag. Hence we obtain a measurable maps
$Z^e=Z^{e,+}$ and $Z^{e,-}$ on $D^\circ(\R,\cD)$ with values in $D_e=D_x([s,\infty),\R)$ by setting
$$
Z^{e,\pm}(\phi)=(\phi_{(s,t]}^\pm(x):t\ge s).
$$
The restrictions of $Z^{e,\pm}$ to $C^\circ(\R,\cD)$ then take values in $C_e=C_x([s,\infty),\R)$.
We define a filtration $(\cF_t)_{t\ge0}$ on $D^\circ(\R,\cD)$ by
$$
\cF_t=\s(Z^e_r:e=(s,x)\in\R^2,r\in(-\infty,t]\cap[s,\infty))
$$
and, for $e=(s,x),e'=(s',x')\in\R^2$, we write $T^{ee'}$ for the collision time
$$
T^{ee'}=\inf\{t\ge s\vee s':Z^e_t-Z^{e'}_t\in2\pi\Z\}.
$$
We make the same definitions for $\phi\in D^\circ(\R,\bar\cD)$, except to define as collision time
$$
\bar T^{ee'}=\inf\{t\ge s\vee s':Z^e_t=Z^{e'}_t\}.
$$

The space $C^\circ(\R,\cD)$ is a convenient state-space for the coalescing Brownian flow on the circle
where it has the following characterization (see \cite[Theorem 6.1]{NT1}).
{\em There exists a unique Borel probability measure $\mu_A$ on $C^\circ(\R,\cD)$
such that, for all $e=(s,x),e'=(s',x')\in\R^2$, the processes $(Z^e_t)_{t\ge s}$ and $(Z^e_tZ^{e'}_t-(t-T^{ee'})^+)_{t\ge s\vee s'}$ are
continuous local martingales in the filtration $(\cF_t)_{t\in\R}$.
Moreover, for all $e\in\R^2$, we have, $\mu_A$-almost surely, $Z^{e,+}=Z^{e,-}$.}

Similarly, the space $C^\circ(\R,\bar\cD)$ is a state-space for the coalescing Brownian flow (on the line).
{\em There exists a unique Borel probability measure $\bar\mu_A$ on $C^\circ(\R,\bar\cD)$
such that, for all $e=(s,x),e'=(s',x')\in\R^2$, the processes $(Z^e_t)_{t\ge s}$ and $(Z^e_tZ^{e'}_t-(t-\bar T^{ee'})^+)_{t\ge s\vee s'}$ are
continuous local martingales in the filtration $(\cF_t)_{t\in\R}$.
Moreover, for all $e\in\R^2$, we have, $\bar\mu_A$-almost surely, $Z^{e,+}=Z^{e,-}$.}

We consider a limit where $f$ becomes an increasingly well-localized perturbation of the identity map.
We quantify this localization in terms of the smallest constant $\l=\l(f,\ve)\in(0,1]$ such that
$$
\frac\rho{2\pi}\int_0^{2\pi}|f_0(x+a)f_0(x)|dx\le\l,\q a\in[\ve\l,2\pi-\ve\l].
$$
We can now state Theorem 6.1 from \cite{NT1}. We have
\begin{equation}\label{WCDF}
\mu_A^f\to\mu_A\q\text{\em weakly on $D^\circ(\R,\cD)$ uniformly in $f\in\cD^*$ as $\rho(f)\to\infty$ and $\l(f,1)\to0$}
\end{equation}
and
\begin{align}\notag
\mu_A^{f,\ve}\to\bar\mu_A\q
\text{\em weakly on $D^\circ(\R,\bar\cD)$ uniformly in } & f\in\cD^*
\text{\em as $\ve\to0$}\\ \label{WCRDF}
&
\text{\em with $\ve^3\rho(f)\to\infty$ and $\l(f,\ve)\to0$}.
\end{align}

\section{The harmonic measure flow}\label{SLAM}
We return to the aggregation model.
We assume throughout this section that \bb condition \eqref{D13} holds. \eb
The boundary $\pd K_n$ of the cluster $K_n$ has a canonical parametrization by $[0,2\pi)$ given by $\th\mapsto\Phi_n(e^{i\th})$.
For $\th_1<\th_2$, the normalized harmonic measure (from $\infty$) of the positively oriented boundary segment from
$\Phi_n(e^{i\th_1})$ to $\Phi_n(e^{i\th_2})$ is then $(\th_2-\th_1)/(2\pi)$.
We consider the related parametrization $\th\mapsto\tilde\Phi_n(i\th):\R\to\pd\tilde K_n$.
For $m\le n$, each point $z\in\pd\tilde K_n$ has a unique {\em ancestor point} $A_{mn}(z)\in\pd\tilde K_m$, which is either $z$ itself
or the point of $\pd\tilde K_m$ to which the particle containing $z$ is attached, possibly through several generations.
On the other hand, each point in $z\in\pd\tilde K_m$, except those points where particles are attached, has a unique {\em escape point}
$E_{mn}(z)\in\pd\tilde K_n$, which is either $z$ itself or is connected to $z$ by a minimal path in $\tilde K_n$, subject to not crossing any particles
nor passing through any attachment points. If $P$ is attached at a single point, then $E_{mn}(z)=z$ for all $z\in\pd\tilde K_m$. \bb These definitions are illustrated in Figure \ref{ancescapefig}. \eb

\begin{figure}[ht]
\centering
\epsfig{file=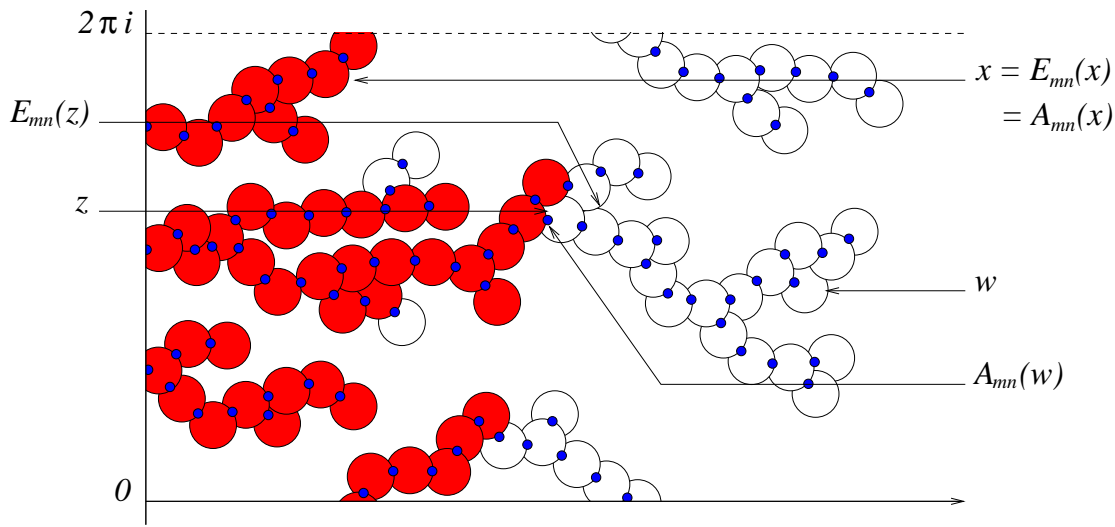,width=15cm}
\caption{\textsl{ Diagram illustrating ancestor points $A_{mn}(z)\in\pd\tilde K_m$ for $z\in\pd\tilde K_n$ and escape points $E_{mn}(z)\in\pd\tilde K_n$ for $z\in\pd\tilde K_m$, where $\tilde K_m$ is shown in red, $\tilde K_n \sm \tilde K_m$ is shown in white, and attachment points are shown in blue.}}
\label{ancescapefig}
\end{figure}

We define the {\em forwards} and {\em backwards harmonic measure flows} on $\R$, respectively, for $0\le m<n$ by
\bb
\begin{equation}\label{HMFG}
\Phi^P_{nm}(x)=-i\tilde\G_n\circ E_{nm}\circ\tilde\Phi_m(ix),\q
\Phi^P_{mn}(x)=-i\tilde\G_m\circ A_{mn}\circ\tilde\Phi_n(ix).
\end{equation}
\eb
We shall show that, when embedded suitably in continuous-time,
these flows converge weakly to the coalescing Brownian flow, as the diameter $\d$ of the basic
particle $P$ tends to $0$.
Then, in the same limiting regime, we shall deduce the behaviour of fingers and gaps in the aggregation model.

First we give an alternative presentation of the flows.
Recall the functions $g^+$ and $f^+$ defined at (\ref{DGF}) and write $g^-$ and $f^-$ for their left-continuous versions.
Then $g=\{g^-,g^+\}\in\cD$ and $f=\{f^-,f^+\}=g^{-1}$.
Since $P$ is non-empty and is invariant under conjugation, $g$ is not the identity function but is an odd function.
Hence $g\in\cD^*$.
Recall that the sequence of clusters $(K_n:n\ge0)$ is constructed from a sequence of independent random variables
$(\Th_n:n\in\N)$, uniformly distributed on $[0,2\pi)$.
Define $f_\th,g_\th\in\cD$ for $\th\in[0,2\pi)$ as in Section \ref{WC}.
Then define for $0\le m<n$
\bb
$$
\Phi^{P,\pm}_{nm}=g^\pm_{\Th_n}\circ\dots\circ g^\pm_{\Th_{m+1}},\q \Phi^{P,\pm}_{mn}=f^\pm_{\Th_{m+1}}\circ\dots\circ f^\pm_{\Th_n}.
$$
\eb
We can check (just as for the disturbance flow) that, almost surely, $\Phi^P_{nm}=\{\Phi^{P,-}_{nm},\Phi^{P,+}_{nm}\}\in\cD$
and $\Phi^P_{mn}=\{\Phi^{P,-}_{mn},\Phi^{P,+}_{mn}\}\in\cD$, with $(\Phi^P_{mn})^{-1}=\Phi^P_{nm}$. Moreover, a straightforward induction shows
that this definition agrees with the more geometric formulation in (\ref{HMFG}).

In formulating a limit statement, it is convenient to embed the harmonic measure flow in continuous time. We do this in two ways.
For a bounded interval $I\sse[0,\infty)$,
set $\Phi^P_I=\Phi^P_{nm}$, where $m+1$ and $n$ are respectively the smallest and largest integers in $\rho I$.
We set $\Phi^P_I=\id$ if there are no such integers. Then $(\Phi^P_I:I\sse[0,\infty))$ takes values in $D^\circ([0,\infty),\cD)$.
Set $\d^*=(\rho c)^{-1}$ and define $\bar\Phi^P_I(x)=(\d^*)^{-1/2}\Phi^P_{\bar n\bar m}((\d^*)^{1/2}x)$, where $\bar m+1$ and $\bar n$
are the smallest and largest integers in $c^{-1}I$. Then $(\bar\Phi^P_I:I\sse[0,\infty))$ takes values in $D^\circ([0,\infty),\bar\cD)$.

\begin{theorem}\label{HMFL}
\bg
Assume that the basic particle $P$ satisfies condition (\ref{D13}).
\eg
Then the harmonic measure flow $(\Phi^P_I:I\sse[0,\infty))$ converges weakly in $D^\circ([0,\infty),\cD)$ to the
coalescing Brownian flow on the circle, uniformly in $P$ as $\d\to0$.
Moreover, the rescaled harmonic measure flow $(\bar\Phi^P_I:I\sse[0,\infty))$ converges weakly in $D^\circ([0,\infty),\bar\cD)$ to the
coalescing Brownian flow on the line.
\end{theorem}
\begin{proof}
The flow $(\Phi^P_I:I\sse[0,\infty))$ is a disturbance flow with disturbance $g$
and $(\bar\Phi^P_I:I\sse[0,\infty))$ is an $\ve$-scale disturbance flow with disturbance $g$, where $\ve=\sqrt{\d^*}$.
From Corollary \ref{GZEE} we know that $\d^2/6\le c\le 3\d^2/4$ and from Proposition \ref{TGES}, we have $\d^{-3}/C\le\rho\le C\d^{-3}$.
Hence $\d^*=(\rho c)^{-1}$ satisfies $\d/C\le\d^*\le C\d$ for an absolute constant $C<\infty$.
In particular $\rho\to\infty$ and $\ve\to0$ and $\ve^3\rho\ge\d^{-3/2}/C\to\infty$ as $\d\to0$.
Also, from Proposition \ref{TGES}, for $a\in[\d,\pi]$, we have
$$
\frac\rho{2\pi}\int_0^{2\pi}|g_0(\th)g_0(\th+a)|d\th\le\frac{C\d}a\log\left(\frac1\d\right)
$$
so $\l(g,1)\le\l(g,\ve)\to0$ as $\d\to0$.
The result thus follows from (\ref{WCDF}) and (\ref{WCRDF}).
\end{proof}

We can now deduce the limiting joint distribution of fingers and gaps.
Recall that $\cS$ denotes the space of locally compact subsets of $[0,\infty)\times\R$, equipped with a local Hausdorff metric.
We have fixed $T>0$ and a finite subset $E$ of $[0,T]\times\R$.
Recall that we study the cluster $K_N$ and have introduced in Section \ref{illus} associated path-like random sets $\fing(z)$ and
$\gap(z)$, along with rescaled sets $\F(e),\bF(e),\GG(e)$ and $\bGG(e)$.
Write $\mu^P_E$ for the law of $(\F(e),\GG(e):e\in E)$ when $N=\lfloor\rho T\rfloor$, considered as a random variable in $(\cS^E)^2$.
Similarly, write $\bar\mu^P_E$ for the law of $(\bF(e),\bGG(e):e\in E)$ when $N=\lfloor c^{-1}T\rfloor$.
Write $\mu_E$ for the law on $(\cS^E)^2$ of the family of random sets
$$
(\{(t,\Phi_{ts(e)}(x(e))):t\in[0,s(e)]\},\{(t,\Phi_{t\wedge T,s(e)}(x(e))):t\ge s(e)\}:e\in E)
$$
where $\Phi$ is a coalescing Brownian flow on the circle and where we set $\Phi_{st}=\Phi_{ts}^{-1}$ for $s\le t$.
Write also $\bar\mu_E$ for the corresponding law when we replace $\Phi$ by a coalescing Brownian flow $\bar\Phi$ on the line.

\begin{theorem}\label{FLOWFING}
\bg
Assume that the basic particle $P$ satisfies conditions (\ref{D13}),(\ref{D14}) and (\ref{HMCO}).
\eg
Then $\mu_E^P\to\mu_E$ and $\bar\mu_E^P\to\bar\mu_E$ weakly on $(\cS^E)^2$, uniformly in $P$ as $\d\to0$.
\end{theorem}
\begin{proof}
We consider first the long time case.
Given $\ve_0>0$, there exist $\ve>0$ and $\ve'\in(0,\ve/3]$ such that, for any coalescing Brownian flow $\Phi=(\Phi_{ts}:0\le s\le t\le T)$ on the circle,
with probability exceeding $1-\ve_0/3$,
for all $e\in E$ and all $t\in[0,T]$, we have
$$
\Phi_{ts(e)}(x(e))-\ve_0\le\Phi_{ts(e)}(x(e)-5\ve)-5\ve,\q\Phi_{ts(e)}(x(e)+5\ve)+5\ve\le\Phi_{ts(e)}(x(e))+\ve_0
$$
and, for all $s,s',t,t'\in[0,T]$ with $|s-s'|,|t-t'|\le3\ve'$ and all $x\in\R$
$$
\Phi_{ts}(x)\le\Phi_{t's'}(x+\ve)+\ve.
$$
\bb Note that these conditions imply $5\ve \le \ve_0$. \eb
Here we have used some standard estimates for Brownian motion and the fact that
the map $(s,t)\mapsto\Phi_{ts}:[0,T]^2\to\cD$ is uniformly continuous, almost surely.
Here and below such inequalities are each to be understood as a pair of inequalities,
one for left-continuous versions and the other for right-continuous versions.

Then, by Theorem \ref{HMFL}, and using a standard result on weak convergence,
there exists a $\d_0>0$ such that,
for all $\d\in(0,\d_0]$ and all basic particles $P$ satisfying \bb \eqref{D13}\eb, 
for $N=\lfloor\rho T\rfloor$, we can construct, on some probability space,
an $HL(0)$ process $\Phi^P=(\Phi_n^P:n\le N)$ with basic particle $P$ and a coalescing Brownian flow $\Phi=(\Phi_{ts}:0\le s\le t\le T)$
on the circle with the following property.
With probability exceeding $1-\ve_0/3$,
for all $0\le m<n\le N$, for $t=m/\rho$ and $s=n/\rho$, and for all $x\in\R$, we have
$$
\Phi_{ts}(x-\ve)-\ve\le\Phi^P_{mn}(x)\le\Phi_{ts}(x+\ve)+\ve.
$$
Here $(\Phi_{mn}^P:0\le m<n\le N)$ is the backwards harmonic measure flow
of $\Phi^P$ (which determines $(\Th_n:1\le n\le N)$ and hence $\Phi^P$ uniquely).

Moreover, by Theorem \ref{BALLG}, we may choose $\d_0$ so that,
with probability exceeding $1-\ve_0/3$,
for all $e\in E$, writing $z(e)=\s^{-1}(e)=s(e)/\d^*+ix(e)$ and $\s(p_0(z(e)))=(s_0,x_0)$, we have
$$
|s_0-s(e)|\le\ve'/3,\q |x_0-x(e)|\le\ve
$$
and,
for all $(s,x)\in[0,T]\times\R$,
there exists $w\in\tilde K_N$ such that $\s(w)=(t,y)$ satisfies
$$
|s-t|\le\ve'/3,\q |x-y|\le\ve
$$
and, for all $n\le N-1$ and all $z\in\tilde P_{n+1}$, $\s(z)=(s,x)$ satisfies
$$
|s-n/\rho|\le\ve'/3,\q |x-\Th_{n+1}|\le\ve.
$$

From this point on, we condition on the good event $\O_0$ of probability exceeding $1-\ve_0$ where all
of the properties discussed above hold.
Suppose that we fix $j,k\in\Z$ and $m,n\le N-1$ and $w\in\tilde P_{m+1}+2\pi ij$ and $z\in\tilde P_{n+1}+2\pi ik$,
with $\tilde P_{m+1}+2\pi ij$ an ancestor particle of $\tilde P_{n+1}+2\pi ik$.
Write $\s(w)=(t,y)$ and $\s(z)=(s,x)$.
Then we must have $m=\rho t'\le n=\rho s'$, with $|s-s'|,|t-t'|\le\ve'/3$ and $|y-(\Th_{m+1}+2\pi j)|,|x-(\Th_{n+1}+2\pi k)|\le\ve$.
Now $\Phi_{mn}^P$ is continuous and
$$
\Th_{m+1}+2\pi j=\Phi^P_{mn}(\Th_{n+1}+2\pi k)
$$
so
\begin{align*}
y\le\Th_{m+1}+2\pi j+\ve&=\Phi^P_{mn}(\Th_{n+1}+2\pi k)+\ve\\
&\le\Phi_{t's'}(\Th_{n+1}+2\pi k+\ve)+2\ve\le\Phi_{t's'}(x+2\ve)+2\ve\le\Phi_{ts}(x+3\ve)+3\ve.
\end{align*}
and by a similar argument also $y\ge\Phi_{ts}(x-3\ve)-3\ve$.
Here we have extended $\Phi$ by setting \bb$\Phi_{ts}=\Phi_{t\wedge T,s\wedge T}$\eb.

Fix $e\in E$ and $(t,y)\in\F(e)$. Write $(t,y)=\s(w)$ and $\tilde P_0(z(e))=\tilde P_{n+1}+2\pi ik$.
We can choose $z\in\tilde P_0(z(e))$ with $\s(z)=(s,x)$ and $|s-s(e)|\le\ve'/3$ and $|x-x(e)|\le\ve$.
Set $u=t\wedge s(e)$.
Then $w$ and $z$ are related as in the preceding paragraph and
$$
t\le t'+\ve'/3\le s'+\ve'/3\le s+2\ve'/3\le s(e)+\ve'\le s(e)+\ve_0
$$
so $|t-u|\le\ve'$. Hence
$$
y\le\Phi_{ts}(x+3\ve)+3\ve\le\Phi_{us(e)}(x+4\ve)+4\ve\le\Phi_{us(e)}(x(e)+5\ve)+5\ve\le\Phi_{us(e)}(x(e))+\ve_0
$$
and similarly
$$
y\ge\Phi_{us(e)}(x(e))-\ve_0.
$$
Since $(t,y)$ was arbitrary, we have shown that
\begin{align*}
\F(e)&\sse\{(t,y):t\in[0,s(e)]\text{ and }|y-\Phi_{ts(e)}(x(e))|\le\ve_0\}\\
&\q\q\cup\{(t,y):t\in[s(e),s(e)+\ve_0]\text{ and }|y-x(e)|\le\ve_0\}
\end{align*}
and, since $\F(e)$ is a connected set joining $(s,x)$ to the imaginary axis, this implies for the Hausdorff metric $d_H$ that
$$
d_H(\F(e),\{(t,\Phi_{ts(e)}(x(e))):0\le t\le s(e)\})\le2\ve_0.
$$

We complete the proof by obtaining an analogous estimate for $\GG(e)$.
Recall that $\GG(e)=\{\s(p_\t):\t\ge0\}$ where $p=p(z(e))$ is the minimal length gap path
starting from $p_0(z(e))$, the closest point to $z(e)$ which is not in the interior of $\tilde K_N$.
Write $\s(p_0(z(e)))=(s_0,x_0)$.

First we show that minimal gap paths cannot backtrack too much.
Suppose that $t<s(e)-\ve'$ and $p$ makes an excursion left of the line $\{t/\d^*+iy:y\in\R\}$, with endpoints $w^-,w^+$, say.
Then the open line segment $(w^-,w^+)$ must contain a point of $\tilde K_N$, say $w\in\tilde P_{m+1}+2\pi ij$. Set $\s(w)=(t,y)$.
Then, since $p$ cannot cross $\tilde K_N$,
there must exist $z\in\tilde P_{n+1}+2\pi ik$, an ancestor particle of $\tilde P_{m+1}+2\pi ij$, with $\s(z)=(s,x)$, say, and $s\ge s_0$.
But then
$$
s(e)\le s_0+\ve'/3\le s+\ve'/3\le n/\rho+2\ve'/3\le m/\rho+2\ve'/3\le t+\ve'<s(e)
$$
which is impossible. Hence there is no such excursion and so
$$
\GG(e)\sse\{(s,x):s\ge s(e)-\ve',x\in\R\}.
$$

Consider $(t,y)=\s(w)$ with $w\in\tilde P_{m+1}+2\pi ij$ and $m\le N-1$ and $t\ge s(e)-3\ve'$ and $y\ge\Phi_{vs(e)}(x(e))+\ve_0$,
where $v=s(e)\vee t\wedge T$. Note that $t\le T+\ve'/3$ and $|v-t|\le3\ve'$.
Suppose $(s,x)=\s(z)$ with $z\in\tilde P_{n+1}+2\pi ik$ and $|s-s(e)|\le\ve'$ and where $\tilde P_{n+1}+2\pi ik$ is an ancestor particle of $\tilde P_{m+1}+2\pi ij$.
Then
$$
x\ge\Phi_{st}(y-3\ve)-3\ve\ge\Phi_{s(e)v}(y-4\ve)-4\ve\ge x(e)+\ve.
$$
Hence $\F(t,y)$ does not meet the vertical half-line $\{(s_0,x):x\le x_0\}$.

Define
$$
\Phi(e)=\{(t,\Phi_{t\wedge T,s(e)}(x(e))):t\ge s(e)\}
$$
and set $I=[s(e)-2\ve',T]$.
There exists a continuous function $(y(t):t\in I)$ such that, for all $t\in I$, setting $v=s(e)\vee t\wedge T$, we have
$$
y(t)>\Phi_{vs(e)}(x(e)),\q d((t,y(t)),\Phi(e))=\ve_0+\ve+5\ve'.
$$
Define recursively a sequence $\t_0,\dots,\t_M$ by setting $\t_0=s(e)-2\ve'$ and then taking $\t_{n+1}$ as the supremum
of the set
$$
\bb
\{\t\in [\tau_n,T]:|(\t,y(\t))-(\t_n,y(\t_n))|=\ve'\}
\eb
$$
until $n=M-1$ when this set is empty and we set $\t_M=T$.
For $n=0,1,\dots,M$, choose $w_n\in\tilde K_N$ with $\s(w_n)=(t_n,y_n)$ and $|t_n-\t_n|\le\ve'$ and $|y_n-y(\t_n)|\le\ve$.
Note that $t_0\le s(e)-\ve'$ and $t_M\ge T-\ve'$ and $t_n\in[s(e)-3\ve',T+\ve']$ for all $n$.
Set
\bb
$$
B_0=\bigcup_{n=0}^{M-1}[w_n,w_{n+1}),\q B_1=\{t/\d^*+iy_M:t\ge t_M\},\q B=B_0\cup B_1.
$$
Then\eb,  for any $w\in B_0$, for $(t,y)=\s(w)$,
we have $|(t,y)-(\t_n,y(\t_n))|\le\ve+2\ve'$ for some $n$, so $\ve_0+3\ve'\le d((t,y),\Phi(e))\le\ve_0+2\ve+7\ve'$
and so $y\ge\Phi_{vs(e)}(x(e))+\ve_0$,
where $v=s(e)\vee t\wedge T$. The final inequality obviously extends to $B$.

Suppose $p$ crosses $B$, and does so for the first time at $\bb \t(1) \eb$. Consider first the case where $p_{\t(1)}\in[w_n,w_{n+1})$.
Then, since $w_n$ and $w_{n+1}$ are both connected to the imaginary axis in $\tilde K_N$ and $p$ cannot cross $\tilde K_N$,
it must eventually hit $[w_n,w_{n+1}]$ again after $\t(1)$, at time $\bb \t(2) \eb$ say, except possibly if $p_{\t(1)}=w_n$.
If the open line segment $(p_{\t(1)},p_{\t(2)})$ contains a point $w\in\tilde K_N$ with $\s(w)=(t,y)$, then for all $z\in\fing(w)$ with $\s(z)=(s,x)$
and $|s-s(e)|\le\ve'$ we have $x\ge x(e)+\ve'$. But this is impossible because $w$ is disconnected from the imaginary axis
by $\{s_0/\d^*+ix:x\le x_0\}\cup\{p_\t:\t\ge0\}$. Hence $(p_{\t(1)},p_{\t(2)})\sse\tilde D_N$, so $p_\t\in[p_{\t(1)},p_{\t(2)}]$ for all $\t\in(\t(1),\t(2))$,
contradicting our crossing assumption. In the case $p_{\t(1)}=w_n$, if $p$ does not return to $[w_n,w_{n+1}]$, then it must hit $[w_{n-1},w_n]$ instead
and this also leads to a contradiction by a similar argument. The case where $p_{\t(1)}\in B_1$ also leads to a contradiction
of minimality by a similar argument.
Hence $p$ never crosses $B$.
So, for all $(t,y)\in\GG(e)$ with $y\ge\Phi_{vs(e)}(x(e))$, we have $d((t,y),\Phi(e))\le\ve_0+2\ve+7\ve'\le2\ve_0$.
A similar argument establishes this estimate also in the case $\bb y\le\Phi_{vs(e)}(x(e)) \eb$.
Since $\GG(e)$ is a connected set joining $(s_0,x_0)$ to $\{T\}\times\R$, this implies
$$
d_H(\GG(e),\Phi(e))\le2\ve_0.
$$

We turn now to the local fluctuations. The argument is mainly similar. It becomes crucial that Theorem \ref{BALLG}
provides approximation on a scale just larger than $\d^{2/3}$, allowing us to transfer fluctuation
results from Theorem \ref{HMFL} at scale $\d^{1/2}$ to the cluster. There is also some loss of compactness
in the local limit which requires attention.

Given $\bb0<\ve_0<1/3\eb$, there exist $\ve>0$ and $R\in[1,\infty)$ and $\ve'\in(0,\ve/3]$ such that,
for any coalescing Brownian flow $\bar\Phi=(\bar\Phi_{ts}:0\le s\le t\le T)$ on the line,
with probability exceeding $1-\ve_0/3$,
for all $e\in E$ and all $t\in[0,T]$, we have
$$
|\bar\Phi_{ts(e)}(x(e))|\le R
$$
and
$$
\bar\Phi_{ts(e)}(x(e))-\ve_0\le\bar\Phi_{ts(e)}(x(e)-5\ve)-5\ve,\q\bar\Phi_{ts(e)}(x(e)+5\ve)+5\ve\le\bar\Phi_{ts(e)}(x(e))+\ve_0
$$
and, for all $s,s',t,t'\in[0,T]$ with $|s-s'|,|t-t'|\le3\ve'$ and all $|x|\le2R$
$$
\bar\Phi_{ts}(x)\le\bar\Phi_{t's'}(x+\ve)+\ve.
$$
Uniform continuity of the map $(s,t)\mapsto\bar\Phi_{ts}:[0,T]^2\to\bar\cD$ now provides only local estimates in $x$,
hence the need for the cut-off $R$.

Then, by Theorem \ref{HMFL},
there exists a $\d_0>0$ such that,
for all $\d\in(0,\d_0]$ and all basic particles $P$ satisfying \bb \eqref{D13}, \eb 
for $N=\lfloor c^{-1}T\rfloor$, we can construct, on some probability space,
an $HL(0)$ process $\Phi^P=(\Phi_n^P:n\le N)$ with basic particle $P$ and a coalescing Brownian flow $\bar\Phi=(\bar\Phi_{ts}:0\le s\le t\le T)$
on the line with the following property.
Write $(\Phi_{mn}^P:0\le m<n\le N)$ for the backwards harmonic measure flow of $\Phi^P$
and set $\bar\Phi_{mn}^P(x)=(\d^*)^{-1/2}\Phi_{mn}^P((\d^*)^{1/2}x)$.
With probability exceeding $1-\ve_0/3$,
for all $0\le m<n\le N$, for $t=cm$ and $s=cn$, and for all $|x|\le2R$, we have
$$
\bar\Phi_{ts}(x-\ve)-\ve\le\bar\Phi^P_{mn}(x)\le\bar\Phi_{ts}(x+\ve)+\ve.
$$

Moreover, by Theorem \ref{BALLG}, we may choose $\d_0$ so that,
with probability exceeding $1-\ve_0/3$,
for all $e\in E$, writing $z(e)=\bar\s^{-1}(e)=s(e)+i(\d^*)^{1/2}x(e)$ and $\bar\s(p_0(z(e)))=(s_0,x_0)$, we have
$$
|s_0-s(e)|\le\ve'/3,\q |x_0-x(e)|\le\ve
$$
and,
for all $s\in[0,T]$ and all $x\in\R$,
there exists $w\in\tilde K_N$ such that $\bar\s(w)=(t,y)$ satisfies
$$
|s-t|\le\ve'/3,\q |x-y|\le\ve
$$
and, for all $n\le N-1$ and all $z\in\tilde P_{n+1}$, $\bar\s(z)=(s,x)$ satisfies
$$
|s-cn|\le\ve'/3,\q |x-\Th_{n+1}/\sqrt{\d^*}|\le\ve.
$$

From this point on, we condition on the good event $\O_0$ of probability exceeding $1-\ve_0$ where all
of the properties discussed above hold.
Suppose that we fix $j,k\in\Z$ and $m,n\le N-1$ and $w\in\tilde P_{m+1}+2\pi ij$ and $z\in\tilde P_{n+1}+2\pi ik$,
with $\tilde P_{m+1}+2\pi ij$ an ancestor particle of $\tilde P_{n+1}+2\pi ik$.
Write $\bar\s(w)=(t,y)$ and $\bar\s(z)=(s,x)$ and suppose that $|x|+2\ve\le2R$.
Then we must have $m=c^{-1}t'\le n=c^{-1}s'$, with $|s-s'|,|t-t'|\le\ve'/3$ and $|y-(\Th_{m+1}+2\pi j)/\sqrt{\d^*}|,|x-(\Th_{n+1}+2\pi k)/\sqrt{\d^*}|\le\ve$,
so
\bb
\begin{align*}
y&\le(\Th_{m+1}+2\pi j)/\sqrt{\d^*}+\ve=\bar\Phi^P_{mn}((\Th_{n+1}+2\pi k)/\sqrt{\d^*})+\ve\\ &\le\bar\Phi_{t's'}((\Th_{n+1}+2\pi k)/\sqrt{\d^*}+\ve) +2\ve\le\bar\Phi_{t's'}(x+2\ve)+2\ve\le\bar\Phi_{ts}(x+3\ve)+3\ve.
\end{align*}
\eb
and by a similar argument also $y\ge\bar\Phi_{ts}(x-3\ve)-3\ve$.
Here we have extended $\bar\Phi$ by setting $\bb \bar\Phi_{ts}=\bar\Phi_{t\wedge T,s\wedge T} \eb$.

Fix $e\in E$ and $(t,y)\in\bar\F(e)$. Write $(t,y)=\bar\s(w)$ and $\tilde P_0(z(e))=\tilde P_{n+1}+2\pi ik$.
We can choose $z\in\tilde P_0(z(e))$ with $\bar\s(z)=(s,x)$ and $|s-s(e)|\le\ve'/3$ and $|x-x(e)|\le\ve$.
In particular $\bb |x|+2\ve\le|x(e)|+3\ve\le2R \eb$.
Set $u=t\wedge s(e)$.
Then $w$ and $z$ are related as in the preceding paragraph and
$$
t\le t'+\ve'/3\le s'+\ve'/3\le s+2\ve'/3\le s(e)+\ve'\le s(e)+\ve_0
$$
so $|t-u|\le\ve'$. Hence
$$
y\le\bar\Phi_{ts}(x+3\ve)+3\ve\le\bar\Phi_{us(e)}(x+4\ve)+4\ve\le\bar\Phi_{us(e)}(x(e)+5\ve)+5\ve\le\bar\Phi_{us(e)}(x(e))+\ve_0
$$
and similarly
$$
y\ge\bar\Phi_{us(e)}(x(e))-\ve_0.
$$
Since $(t,y)$ was arbitrary, we have shown that
\begin{align*}
\bar\F(e)&\sse\{(t,y):t\in[0,s(e)]\text{ and }|y-\bar\Phi_{ts(e)}(x(e))|\le\ve_0\}\\
&\q\q\cup\{(t,y):t\in[s(e),s(e)+\ve_0]\text{ and }|y-x(e)|\le\ve_0\}
\end{align*}
and, since $\bar\F(e)$ is a connected set joining $(s,x)$ to the imaginary axis, this implies for the Hausdorff metric $d_H$ that
$$
d_H(\bar\F(e),\{(t,\bar\Phi_{ts(e)}(x(e))):0\le t\le s(e)\})\le2\ve_0.
$$

We complete the proof by obtaining an analogous estimate for $\bar\GG(e)$.
Recall that $\bar\GG(e)=\{\bar\s(p_\t):\t\ge0\}$ where $p=p(z(e))$ is the minimal length gap path
starting from $p_0(z(e))$, the closest point to $z(e)$ which is not in the interior of $\tilde K_N$.
Write $\bar\s(p_0(z(e)))=(s_0,x_0)$.

Suppose that $t<s(e)-\ve'$ and $p$ makes an excursion left of the line $\{t+i\sqrt{\d^*}y:y\in\R\}$, with endpoints $w^-,w^+$, say.
Then the open line segment $(w^-,w^+)$ must contain a point of $\tilde K_N$, say $w\in\tilde P_{m+1}+2\pi ij$. Set $\bar\s(w)=(t,y)$.
Then, since $p$ cannot cross $\tilde K_N$,
there must exist $z\in\tilde P_{n+1}+2\pi ik$, an ancestor particle of $\tilde P_{m+1}+2\pi ij$, with $\s(z)=(s,x)$, say, and $s\ge s_0$.
But then
$$
\bb
s(e)\le s_0+\ve'/3\le s+\ve'/3\le cn+2\ve'/3\le cm+2\ve'/3\le t+\ve'<s(e)
\eb
$$
which is impossible. Hence there is no such excursion and so
$$
\bar\GG(e)\sse\{(s,x):s\ge s(e)-\ve',x\in\R\}.
$$

Consider $(t,y)=\bar\s(w)$ with $w\in\tilde P_{m+1}+2\pi ij$ and $m\le N-1$ and $t\ge s(e)-3\ve'$ and \bb$|y|+3\ve \le 2R$ and \eb $y\ge\bar\Phi_{vs(e)}(x(e))+\ve_0$,
where $v=s(e)\vee t\wedge T$. Note that $t\le T+\ve'/3$ and $|v-t|\le3\ve'$.
Suppose $(s,x)=\bar\s(z)$ with $z\in\tilde P_{n+1}+2\pi ik$ and $|s-s(e)|\le\ve'$ and where $\tilde P_{n+1}+2\pi ik$ is an ancestor particle of $\tilde P_{m+1}+2
\pi ij$.
Then
$$
x\ge\bar\Phi_{st}(y-3\ve)-3\ve\ge\bar\Phi_{s(e)v}(y-4\ve)-4\ve\ge x(e)+\ve.
$$
Hence $\bar\F(t,y)$ does not meet the vertical half-line $\{(s_0,x):x\le x_0\}$.

Define
$$
\bar\Phi(e)=\{(t,\bar\Phi_{t\wedge T,s(e)}(x(e))):t\ge s(e)\}
$$
and set $I=[s(e)-2\ve',T]$.
There exists a continuous function $\bb y(t):I\to\R \eb$ such that, for all $t\in I$, setting $v=s(e)\vee t\wedge T$, we have
$$
y(t)>\bar\Phi_{vs(e)}(x(e)),\q d((t,y(t)),\bar\Phi(e))=\ve_0+\ve+5\ve'.
$$
Define recursively a sequence $\t_0,\dots,\t_M$ by setting $\t_0=s(e)-2\ve'$ and then taking $\t_{n+1}$ as the supremum
of the set
$$
\bb
\{\t\in [\t_n,T]:|(\t,y(\t))-(\t_n,y(\t_n))|=\ve'\}
\eb
$$
until $n=M-1$ when this set is empty and we set $\t_M=T$.
For $n=0,1,\dots,M$, choose $w_n\in\tilde K_N$ with $\bar\s(w_n)=(t_n,y_n)$ and $|t_n-\t_n|\le\ve'$ and $|y_n-y(\t_n)|\le\ve$.
Note that $t_0\le s(e)-\ve'$ and $t_M\ge T-\ve'$ and $t_n\in[s(e)-3\ve',T+\ve']$ \bb and $|y_n| + 3\ve \le 2R$ \eb  for all $n$.
Set
\bb
$$
B_0=\bigcup_{n=0}^{M-1}[w_n,w_{n+1}),\q B_1=\{t+i\sqrt{\d^*}y_M:t\ge t_M\},\q B=B_0\cup B_1.
$$
Then\eb,  for any $w\in B_0$, for $(t,y)=\bar\s(w)$,
we have $|(t,y)-(\t_n,y(\t_n))|\le\ve+2\ve'$ for some $n$, so $\ve_0+3\ve'\le d((t,y),\bar\Phi(e))\le\ve_0+2\ve+7\ve'$
and so $y\ge\bar\Phi_{vs(e)}(x(e))+\ve_0$,
where $v=s(e)\vee t\wedge T$. The final inequality obviously extends to $B$.

Suppose $p$ crosses $B$, and does so for the first time at $\bb \t(1) \eb$. Consider first the case where $p_{\t(1)}\in[w_n,w_{n+1})$.
Then, since $w_n$ and $w_{n+1}$ are both connected to the imaginary axis in $\tilde K_N$ and $p$ cannot cross $\tilde K_N$,
it must eventually hit $[w_n,w_{n+1}]$ again after $\t(1)$, at time $\bb \t(2) \eb$ say, except possibly if $p_{\t(1)}=w_n$.
If the open line segment $(p_{\t(1)},p_{\t(2)})$ contains a point $w\in\tilde K_N$ with $\bar\s(w)=(t,y)$, then for all $z\in\fing(w)$ with $\bar\s(z)=(s,x)$
and $|s-s(e)|\le\ve'$ we have $x\ge x(e)+\ve'$. But this is impossible because $w$ is disconnected from the imaginary axis
by $\{s_0+i\sqrt{\d^*}x:x\le x_0\}\cup\{p_\t:\t\ge0\}$. Hence $(p_{\t(1)},p_{\t(2)})\sse\tilde D_N$, so $p_\t\in[p_{\t(1)},p_{\t(2)}]$ for all $\t\in(\t(1),\t(2))$,
contradicting our crossing assumption. In the case $p_{\t(1)}=w_n$, if $p$ does not return to $[w_n,w_{n+1}]$, then it must hit $[w_{n-1},w_n]$ instead
and this also leads to a contradiction by a similar argument. The case where $p_{\t(1)}\in B_1$ also leads to a contradiction
of minimality by a similar argument.
Hence $p$ never crosses $B$.
So, for all $(t,y)\in\bar\GG(e)$ with $y\ge\bar\Phi_{vs(e)}(x(e))$, we have $d((t,y),\bar\Phi(e))\le\ve_0+2\ve+7\ve'\le2\ve_0$.
A similar argument establishes this estimate also in the case $\bb y\le\bar\Phi_{vs(e)}(x(e)) \eb$.
Since $\bar\GG(e)$ is a connected set joining $(s_0,x_0)$ to $\{T\}\times\R$, this implies
$$
d_H(\bGG(e),\bar\Phi(e))\le2\ve_0.
$$
\end{proof}

\bibliography{bibliography}
\end{document}